\documentclass[12pt]{amsart}

\usepackage[all]{xy}
\usepackage{amscd,amsmath,amsthm,amssymb}
\usepackage{mathtools}
\usepackage{mathrsfs}
\usepackage{tikz}
\usepackage{caption}

\definecolor{cadmiumgreen}{rgb}{0.0, 0.42, 0.24}
\usepackage[
colorlinks, citecolor=cadmiumgreen,
pdfauthor={Robin de Jong, Farbod Shokrieh}, 
pdfstartview ={FitV},
]{hyperref}
\hypersetup{
pdftitle={Jumps in the height of the Ceresa cycle},
}

\usepackage[
msc-links,
nobysame,
lite,
]{amsrefs} 

\usepackage[normalem]{ulem}

\usepackage[left=3.5cm,top=3.5cm,right=3.5cm]{geometry}
\usepackage{enumitem} 

\newtheorem{thm}{Theorem}[section]
\newtheorem{prop}[thm]{Proposition}

\newtheorem{lem}[thm]{Lemma}

\theoremstyle{remark}
\newtheorem{remark}[thm]{Remark}
\newtheorem{example}[thm]{Example}

\theoremstyle{definition}
\newtheorem{definition}[thm]{Definition}

\makeatletter
\renewcommand*\env@matrix[1][*\c@MaxMatrixCols c]{
  \hskip -\arraycolsep
  \let\@ifnextchar\new@ifnextchar
  \array{#1}}
\newcommand*\isom{\xrightarrow{\sim}}
\newcommand{\divisor}{\operatorname{div}}
\newcommand{\ord}{\operatorname{ord}}

\newcommand{\Hom}{\operatorname{Hom}}

\newcommand{\Mat}{\operatorname{Mat}}

\newcommand{\Ext}{\operatorname{Ext}}

\newcommand{\Spec}{\operatorname{Spec}}

\newcommand{\Jac}{\operatorname{Jac}}

\newcommand{\Sp}{\operatorname{Sp}}

\newcommand{\Hdg}{\operatorname{Hdg}}

\def\opn#1#2{\def#1{\operatorname{#2}}}
\opn\Vor{Vor}

\def\qq{\mathbb{Q}}

\def\rr{\mathbb{R}}

\def\zz{\mathbb{Z}}
\def\cc{\mathbb{C}}
\def\gg{\mathbb{G}}

\def\UU{\mathcal{U}}

\def\mm{\mathcal{M}}
\def\ll{\mathcal{L}}
\def\bb{\mathcal{B}}

\def\CC{\mathcal{C}}
\def\jj{\mathcal{J}}

\def\aa{\mathcal{A}}
\def\oo{\mathcal{O}}

\def\dd{\mathbb{D}}

\def\ee{\mathcal{E}}

\def\Im{\mathrm{Im}\,}
\def\d{\mathrm{d}}

\def\sing{\mathrm{sing}}

\def\id{\mathrm{id}}

\def\can{\mathrm{can}}

\def\Gr{\mathrm{Gr}}

\def\Id{\mathrm{Id}}
\def\swt{\boldsymbol{\sigma}}

\numberwithin{equation}{section}

\subjclass[2010]{
\href{https://mathscinet.ams.org/msc/msc2010.html?t=14C30}{14C30},
\href{https://mathscinet.ams.org/msc/msc2010.html?t=14D07}{14D07},
\href{https://mathscinet.ams.org/msc/msc2010.html?t=14H15}{14H15},
\href{https://mathscinet.ams.org/msc/msc2010.html?t=32G20}{32G20}}

\begin{document}

\title[Jumps in the height of the Ceresa cycle]{Jumps in the height of the Ceresa cycle}

\author{Robin de Jong}
\address{Leiden University \\ PO Box 9512  \\ 2300 RA Leiden  \\ The Netherlands }
\email{\href{mailto:rdejong@math.leidenuniv.nl}{rdejong@math.leidenuniv.nl}}

\author{Farbod Shokrieh}
\address{University of Washington \\ Box 354350 \\ Seattle, WA 98195 \\ USA }
\email{\href{mailto:farbod@uw.edu}{farbod@uw.edu}}

\begin{abstract} We study the jumps in the archimedean height of the Ceresa cycle, as introduced by R.\ Hain in his work on normal functions on moduli spaces of curves, and as further analyzed by P.\ Brosnan and G.\ Pearlstein in terms of asymptotic Hodge theory. Our work is based on a study of the asymptotic behavior of the Hain-Reed beta-invariant in degenerating families of curves. We show that the height jump of the Ceresa cycle at a given stable curve is equal to the so-called ``slope'' of the dual graph of the curve, and we characterize those stable curves for which the height jump vanishes. We also obtain an analytic formula for the height of the Ceresa cycle for a curve over a function field over the complex numbers, and characterize in analytic terms when the height of the Ceresa cycle vanishes. 
\end{abstract}

\maketitle

\setcounter{tocdepth}{1}

%\tableofcontents

\thispagestyle{empty}

\section{Introduction}
\renewcommand*{\thethm}{\Alph{thm}}

\subsection{Motivation and background} 

The central object of study in this paper is the so-called \emph{Hain-Reed line bundle} $\bb$ on the moduli space of curves $\mm_g$, where $g \geq 2$ is an integer. The line bundle $\bb$ is introduced in \cite{hrar} and comes equipped with a canonical smooth hermitian metric $\|\cdot\|_\bb$. The curvature form of the metric $\|\cdot\|_\bb$  is semi-positive, and the metric gives rise to the archimedean contributions  in computing the Arakelov height of the Ceresa cycle $C - C^-$ in the Jacobian of a genus~$g$ curve~$C$.

A natural question, which forms the point of departure of \cite{hrar},  is to what extent the canonical metric $\|\cdot\|_\bb$ extends over $ \overline{\mm}_g$, the Deligne-Mumford compactification of $\mm_g$. Let $\varDelta_0$ denote the component of the boundary divisor $\varDelta$ of $ \overline{\mm}_g$ whose generic point corresponds to an irreducible stable curve with one node. In their paper, R.~Hain and D.~Reed show that the metric $\|\cdot\|_\bb$ extends in a smooth manner over $\widetilde{\mm}_g = \overline{\mm}_g \setminus \varDelta_0$, the moduli space of curves of compact type, and that  the metric $\|\cdot\|_\bb$ extends continuously over any holomorphic arc in $ \overline{\mm}_g$ that meets the component $\varDelta_0$ transversally. 

Given this result it is reasonable to ask whether the metric $\|\cdot\|_\bb$ extends continuously over the entire $ \overline{\mm}_g$. In the follow-up paper \cite{hain_normal} Hain shows, by means of an example, that this is \emph{not} the case. In order to probe the singularities of the metric $\|\cdot\|_\bb$  near the points of $\varDelta$,  Hain introduces in \cite{hain_normal} what he calls the \emph{height jump} of the metric $\|\cdot\|_\bb$. 

\subsection{The jump in the height of the Ceresa cycle}
When $p \in \overline{\mm}_g$ is a point, Hain's height jump at $p$ is a certain homogeneous weight one element $j(p) \in \qq(x_i \, | \, i \in \mathcal{I})$, where $\mathcal{I}$ is a set indexing the local branches of the boundary divisor $\varDelta$ at $p$. The height jump $j(p)$ at $p$ vanishes if the metric $\|\cdot\|_\bb$ extends continuously in an open neighborhood of~$p$; thus, to give a counterexample to the continuous extendability of the metric $\|\cdot\|_\bb$,  it suffices to find a point $p  \in \overline{\mm}_g$ where the height jump $j(p)$ is non-trivial. 

The height jump $j(p)$  is locally constant, and specializes upon generalization, on the  strata of $\overline{\mm}_g$ determined by the normal crossings divisor~$\varDelta$. In particular, the subset of $\overline{\mm}_g$ where the height jump vanishes is closed under generalization and  determines a canonical open subset $\mm_g^\flat$ of $\overline{\mm}_g$. The result by Hain and Reed shows that $\mm_g^\flat$ contains $\overline{\mm}_g \setminus \varDelta_0^\sing$, where $\varDelta_0^\sing$ denotes the singular locus of $\varDelta_0$.  One aim of the present paper is to give a characterization of $\mm_g^\flat$ in geometrical and combinatorial terms, see Theorem~\ref{cor:vanishing}.

In \cite{hain_normal} Hain stated two conjectures regarding the asymptotic behavior of the metric on the line bundle $\bb$. One is that for each $(m_i)_{i\in \mathcal{I}} \in \qq_{>0}^\mathcal{I}$ the rational number $j(p;(m_i)_{i\in \mathcal{I}})$ is non-negative; the other is that for all smooth connected complex curves $T$ with projective completion $\overline{T}$ and all morphisms $f \colon T \to \mm_g$, the first Chern form $c_1(f^* \bb)$ extends as a semi-positive $(1,1)$-current over $\overline{T}$. Both conjectures have been proved; by P.~Brosnan and G.~Pearlstein in \cite{bp} and independently by J. Burgos Gil, D.~Holmes and the first named author in \cite{bghdj}.

\subsection{The work of Brosnan and Pearlstein} \label{sec:BP_work}
In \cite{bp}, Brosnan and Pearlstein analyze the height jump phenomenon in a broader setting. In fact they study the singularities of the heights of \emph{biextension variations of mixed Hodge structure} in complete generality and relate the general height jump to a canonical pairing, called the \emph{asymptotic height pairing}, on the local intersection cohomology groups of the underlying variations of pure Hodge structure. 

Specializing the general results of \cite{bp} to the setting of the line bundle $\bb$ we find, among many other things, the  following improvements of the extension results of Hain and Reed mentioned above.  
\begin{itemize}
\item[(i)] The line bundle $\bb$ extends uniquely to a hermitian line bundle $\bb^*$ with a continuous metric over $\overline{\mm}_g \setminus \varDelta_0^\sing$;
\item[(ii)] The line bundle $\bb^*$ extends uniquely to a hermitian  line bundle $\overline{\bb}$ with plurisubharmonic local potentials over $\overline{\mm}_g$;
\item[(iii)] Let $p \in \overline{\mm}_g$ be a point and let $j(p) \in \qq(x_i \, | \, i \in \mathcal{I})$ be the jump in the height of the Ceresa cycle at $p$. The following two assertions are equivalent:
\begin{itemize}
\item the height jump $j(p)$ vanishes;
\item the plurisubharmonic metric on $\overline{\bb}$ has bounded local potentials near~$p$.
\end{itemize}
\end{itemize}
Note that item (iii) shows once more that the locus of $\overline{\mm}_g$ where the height jump vanishes is an open set. Item (ii) immediately leads to a proof of the second conjecture of Hain mentioned above.

\subsection{Aim of this paper} 
Let $p \in \overline{\mm}_g$ be any point. The main aim of this paper is to give an explicit combinatorial expression for the height jump $j(p)$ at $p$ in terms of the dual graph of the stable curve $C$ corresponding to $p$. In very brief terms, we show that the height jump equals the so-called \emph{slope} of the dual graph of $C$ -- see Theorem~\ref{main:intro}. 

The slope is a certain invariant of polarized  graphs which appears implicitly in fundamental work of S.~Zhang \cite{zhgs} and Z.~Cinkir \cite{ci}  on the height of the canonical Gross-Schoen cycle for curves over global fields. 

In order to define the slope we first need to introduce the $\lambda$-invariant of a polarized graph.

\subsection{The $\lambda$-invariant of a polarized graph} \label{sec:lambda_graph}
For the notions of \emph{polarized graphs} and \emph{weighted graphs} we refer to \S\ref{pm_graphs}.  We will use the letter $G$ to refer to graphs but also to weighted graphs if the underlying edge lengths are clear. We use the notation $\overline{G}=(G,\mathbf{q})$ for polarized (weighted) graphs.

When $G$ is a connected (unweighted) graph and  $\mathbf{q}$ a polarization of $G$, then the polarized graph $\overline{G}=(G,\mathbf{q})$ has associated to it a \emph{genus} $\overline{g}(\overline{G}) \in \zz_{>0}$. For example, the dual graph of a stable curve of genus~$g \geq 2$ is naturally a connected polarized graph of genus~$g$. 

When $S$ is a connected Dedekind scheme, and $\pi \colon \mathcal{X} \to S$ is a stable curve of genus~$g \geq 2$ with smooth generic fiber, and $s$ is a closed point of $S$, then the dual graph of the fiber $\mathcal{X}_s$ of $\mathcal{X}$ at $s$ has a natural structure of a connected polarized  \emph{weighted} graph of genus~$g$. Here, the lengths of the edges are determined by the ``thicknesses'' of the singular points of $\mathcal{X}_s$ on the surface~$\mathcal{X}$.

Now let $G$ be a connected weighted graph with length function $m$, and let $\mathbf{q}$ be a polarization of $G$. Let $\overline{G}=(G,\mathbf{q})$ denote the resulting polarized weighted graph. Then to $\overline{G}$ Zhang associates in \cite[Section~1.4]{zhgs} a real-valued invariant $\lambda(\overline{G})$. From the definition of $\lambda(\overline{G})$ it is not difficult to see that when $(G,\mathbf{q})$ is a connected polarized unweighted graph with edge set $E$, the map that assigns to each element $m \in \rr_{>0}^E$ the real number $\lambda(G,\mathbf{q};m)$ is represented by a homogeneous weight one element $\lambda(G,\mathbf{q}) \in \qq(x_e \, | \, e \in E)$. 

It follows that we can also talk about the $\lambda$-invariant of a connected polarized graph, treating the edge lengths as variables.

\subsection{Height of the Ceresa cycle in the function field case} \label{sec:height_Cer_FF}
We will not give the precise definition of Zhang's $\lambda$-invariant here in the introduction as it is a little involved, but instead we mention its connection to the computation of the height of the Ceresa cycle in the function field case, following Zhang's paper \cite{zhgs}. 

Let $\overline{S}$ be a smooth projective geometrically connected curve over a field $k$. Let $\pi \colon \mathcal{X} \to \overline{S}$ be a stable curve of genus $g \geq 2$ with smooth generic fiber. The generic fiber of $\mathcal{X}$ can be viewed as a smooth projective geometrically connected curve of genus $g$ over the function field of $\overline{S}$. As such it has a natural associated height of the Ceresa cycle $c(\mathcal{X}/\overline{S})$ and height of the canonical Gross-Schoen cycle $\langle \Delta, \Delta \rangle(\mathcal{X}/\overline{S})$. Both heights are rational numbers, and it follows from  \cite[Theorem~1.5.6]{zhgs} that the two heights are related by the simple relation
\begin{equation} \label{ceresa_gs}
c(\mathcal{X}/\overline{S}) = \frac{2}{3} \langle \Delta, \Delta \rangle (\mathcal{X}/S) \, .
\end{equation}
For each closed point $s \in S$ we denote by $\overline{G}_s$ the polarized weighted graph associated to the fiber of $\pi$ at $s$ as explained above. We write
\begin{equation}
 \lambda(\mathcal{X}/\overline{S}) =  \sum_{s \in |\overline{S}|} \lambda(\overline{G}_s) \, . 
\end{equation}
This is a finite sum, as the $\lambda$-invariant of a point graph is zero.
Let $h(\mathcal{X}/\overline{S}) = \deg \det \pi_* \omega_{\mathcal{X}/\overline{S}}$ be the modular height of the stable curve $\pi$.
A combination of (\ref{ceresa_gs}) and \cite[Equation~(1.4.2)]{zhgs} leads to the fundamental relation
\begin{equation} \label{mod_height:intro}
h(\mathcal{X}/\overline{S})  = \frac{3g-3}{2g+1} c(\mathcal{X}/\overline{S}) +   \lambda(\mathcal{X}/\overline{S}) 
\end{equation}
between the modular height $h(\mathcal{X}/\overline{S}) $, the height of the Ceresa cycle $c(\mathcal{X}/\overline{S})$ and the $\lambda$-invariants of the closed fibers of $\pi$.

\subsection{The slope of a polarized graph} 
Let $\overline{G} = (G,\mathbf{q})$ be a connected polarized weighted graph of genus $g \geq 2$. We denote by $\delta_0(G)$ the total length of the edges of $G$ that do not disconnect the graph $G$ upon removal. For $h \in \{1,\ldots,[g/2] \}$ we denote by $\delta_h(\overline{G})$ the total length of the edges of $G$ whose removal from $G$ results in the disjoint union of a polarized graph of genus $h$ and a polarized graph of genus $g-h$. 

\medskip

\noindent \textbf{Definition.} We define the \emph{slope} of the polarized weighted graph $\overline{G} = (G,\mathbf{q})$  to be the real number
\begin{equation} \label{def:slope}
 s(\overline{G}) = (8g+4) \, \lambda(\overline{G}) - g \, \delta_0(G) - \sum_{h=1}^{[g/2]} 4h(g-h) \, \delta_h(\overline{G}) \, . 
\end{equation}

\medskip 

The terminology ``slope'' is explained by the following relation with the celebrated ``slope inequality'' found by A.\ Moriwaki  \cite[Theorem~D]{moriw}. 

\subsection{Moriwaki's slope inequality} \label{sec:mori}
As in \S\ref{sec:height_Cer_FF} let $\overline{S}$ be a smooth projective geometrically connected curve over a field $k$, and let $\pi \colon \mathcal{X} \to \overline{S}$ be a stable curve of genus $g \geq 2$ with smooth generic fiber.  Continuing with the notation introduced in \S\ref{sec:height_Cer_FF} we define
\begin{equation}
 s(\mathcal{X}/\overline{S}) =  \sum_{s \in |\overline{S}|} s(\overline{G}_s) \, , \quad 
 \delta_0(\mathcal{X}/\overline{S}) = \sum_{s \in |\overline{S}|} \delta_0(G_s) \, , \quad
\delta_h(\mathcal{X}/\overline{S}) = \sum_{s \in |\overline{S}|} \delta_h(\overline{G}_s)   
\end{equation}
for $ h=1,\ldots, [g/2]$. From (\ref{mod_height:intro}) and (\ref{def:slope}) we immediately obtain the following relation:
\begin{equation} \label{mod_heightII:intro}
\begin{split} 
(8g+4) h(\mathcal{X}/\overline{S}) - g \,  \delta_0(\mathcal{X}/\overline{S}) - \sum_{h=1}^{[g/2]} & 4h(g-h) \, \delta_h(\mathcal{X}/\overline{S}) \\ & = 12(g-1) c(\mathcal{X}/\overline{S}) +  s(\mathcal{X}/\overline{S}) \, . \end{split} 
\end{equation}
We define $m(\mathcal{X}/\overline{S})$ to be the left hand side of (\ref{mod_heightII:intro}). Assuming that $\mathrm{char}(k)=0$, Moriwaki's slope inequality states that 
\begin{equation}
m(\mathcal{X}/\overline{S}) \geq 0 \, . 
\end{equation}
The following was conjectured by Zhang \cite[Conjecture~1.4.5]{zhgs} and proved by Cinkir \cite[Theorem~2.13]{ci}.

\medskip
\noindent \textbf{Theorem.} For all connected polarized weighted graphs $\overline{G}$ of genus $g\geq 2$ the slope $s(\overline{G})  $ is a non-negative real number. 
\medskip

Cinkir's theorem implies  that the term $s(\mathcal{X}/\overline{S})$ in (\ref{mod_heightII:intro})  is non-negative. Also, it can be proved (see, e.g. \cite[Theorem~1]{zh_hodge} -- still under the assumption that $\mathrm{char}(k)=0$) that the height  $\langle \Delta, \Delta \rangle (\mathcal{X}/\overline{S}) $ of the canonical Gross-Schoen cycle is non-negative. Thus, by (\ref{ceresa_gs}) we have that the height $c(\mathcal{X}/\overline{S}) $ of the Ceresa cycle is non-negative. Combining these two non-negativity results we see that (\ref{mod_heightII:intro}) can be viewed as a refinement of Moriwaki's inequality, as was also observed by Zhang and Cinkir in their work.

\subsection{Main result}

Similar to the $\lambda$-invariant, when $(G,\mathbf{q})$ is a connected polarized unweighted graph, the map that assigns to each $m \in \rr_{>0}^E$ the real number $s(G,\mathbf{q};m)$ is represented by a homogeneous weight one function $s(G,\mathbf{q}) \in \qq(x_e \, | \, e \in E)$. It follows that we can also talk about the slope of a connected polarized graph, treating the edge lengths as variables. The slope will occur as such in our main result, to be discussed next.

 Let $g \geq 2$ be an integer, and let $p \in \overline{\mm}_g$ be a point. Let $C$ be the stable curve of genus~$g$ corresponding to $p$, and let $(G, \mathbf{q})$ be the dual graph of $C$, viewed as a connected polarized graph.
 Let $\mathcal{I}$ be the set of local branches of the boundary divisor $\varDelta$ at $p$, and let $E$ be the set of edges of $G$. As is well known we have a canonical bijection $\mathcal{I} \isom E$. 
 
 Our main result is as follows.
 \begin{thm} \label{main:intro}  The height jump of the Ceresa cycle at the moduli point $p$ equals the slope of the dual graph $(G,\mathbf{q})$. More precisely, under the canonical bijection $\mathcal{I} \isom E$ the elements $j(p) \in \qq(x_i \, | \, i \in \mathcal{I})$ and $s(G,\mathbf{q}) \in \qq(x_e \, | \, e \in E)$ coincide.
\end{thm}

\subsection{An example}
As an illustration of Theorem~\ref{main:intro} we consider the case where $C$ consists of two smooth irreducible components of genera $h$ and $g-h-1$ attached in two points. The stability of $C$ ensures that $(g-h-1)h>0$. In \cite[Theorem~241]{bp} it is shown that the height jump of the Ceresa cycle at the moduli point corresponding to~$C$ in $\overline{\mm}_g$ is equal to
\begin{equation} \label{jump_2_gon}
  j(p)(x_1,x_2) = \frac{4x_1x_2}{x_1+x_2} (g-h-1)h  \, .
\end{equation}
In particular, the height jump is non-trivial in this case.  This example generalizes the example of a non-trivial height jump originally given by Hain in  \cite{hain_normal}. 

Let $(G,\mathbf{q})$ be the dual graph of $C$. Then $G$ consists of two vertices of genera $h$ and $g-h-1$, joined by two edges. From Example~\ref{lambda_for_two_gon} it follows that
\begin{equation}
 s(G,\mathbf{q})(x_1,x_2) =  \frac{4x_1x_2}{x_1+x_2} (g-h-1)h  \, . 
\end{equation}
Thus using Theorem~\ref{main:intro} we are able to reproduce the result in (\ref{jump_2_gon}). 

The proof  of (\ref{jump_2_gon}) in \cite{bp} is based on an analysis of the asymptotic height pairing in terms of a suitable partial Koszul complex computing intersection cohomology together with an
explicit computation featuring Johnson's homomorphism on the Torelli group. 
It seems not straightforward to us  to generalize the strategy leading to \cite[Theorem~241]{bp} to handle other cases.

\subsection{Vanishing of the height jump}
Using Theorem~\ref{main:intro} we can give a complete characterization of when the height jump vanishes.
As above let $p \in \overline{\mm}_g$ be a point and let $C$ be the corresponding stable curve.
\begin{thm} \label{cor:vanishing}  The height jump of the Ceresa cycle at $p$ vanishes if and only if one of the following two conditions holds:
\begin{itemize}
\item[(a)] the curve $C$ consists of two smooth rational irreducible components, joined in $g+1$ points, or 
\item[(b)] the curve $C$ is tree-like, i.e.,   its irreducible components form a tree.
\end{itemize}
\end{thm}
It is clear from Theorem~\ref{main:intro} that proving Theorem~\ref{cor:vanishing} is a matter of characterizing for which polarized weighted graphs $(G,\mathbf{q})$  the slope vanishes, and interpreting the result in geometric terms. This task will be carried out in \S\ref{vanishing_slope} (Theorem~\ref{thm:slopezero}).

\subsection{The Hain-Reed $\beta$-invariant} \label{sec:HRbeta}

Our approach to Theorem~\ref{main:intro} is close in spirit to the original paper \cite{hrar} of Hain and Reed and is based on a study of a certain class of functions $\beta \colon \mm_g \to \rr$ introduced in \cite{hrar}. The jump in the height of the Ceresa cycle can be rewritten in terms  of the asymptotics of $\beta$ near the boundary divisor~$\varDelta$. 

Let $p \colon \CC_g \to \mm_g$ be the universal Riemann surface. As is shown in \cite{hrar}, the underlying line bundle of $\bb$ is isomorphic to $\ll^{\otimes 8g+4}_p$, where $\ll_p =\det p_* \omega_{\CC_g/\mm_g}$  is the determinant of the Hodge bundle on $\mm_g$. We recall that the line bundle $\ll_p$ is equipped with a natural smooth hermitian metric $\| \cdot \|_{\Hdg}$, the Hodge metric, derived from the inner product given by $(\alpha, \beta) \mapsto \frac{i}{2} \int_C \alpha \wedge \overline{\beta}$ on the space of holomorphic $1$-forms on a compact connected Riemann surface $C$. 

As the only invertible holomorphic functions on $\mm_g$ are constants, the set of isomorphisms $\phi \colon \bb \isom \ll_p^{\otimes 8g+4}$ is a non-empty $\cc^\times$-torsor. Picking such an isomorphism $\phi$ allows us to define an $\rr$-valued function
\begin{equation} \label{def:beta}
 \beta_\phi = \log \left( \frac{ \| \cdot \|_\bb }{\phi^* \| \cdot \|_{\Hdg}} \right)
\end{equation}
on $\mm_g$. Here and below, the function log is taken to be the natural logarithm. Varying the isomorphism $\phi$ we obtain a class of functions on $\mm_g$ modulo constants, which is the $\beta$-invariant of Hain and Reed. 

Our main technical result is a result on the asymptotics of the $\beta$-invariant in one-parameter families of degenerating compact Riemann surfaces. We will work with a specific representative of the $\beta$-invariant, which is given by a certain archimedean analogue of the combinatorial $\lambda$-invariant that we saw earlier. The archimedean $\lambda$-invariant is also defined in \cite[Section~1.4]{zhgs}. It is shown in \cite[Theorem~1.4]{dj_second} that, for a suitable choice of the isomorphism $\phi$, one has the equality of functions $\beta_\phi = (8g+4) \lambda$ on $\mm_g$. 

\subsection{Height of the Ceresa cycle in the number field case}
We will not give the precise definition of the archimedean $\lambda$-invariant here as it is again a little involved, but instead we mention its connection to the computation of the height of the Ceresa cycle in the number field setting. The discussion here is completely analogous to the one in \S\ref{sec:height_Cer_FF} which deals with the function field setting.

Let $k$ be a number field, and write $S= \Spec O_k$ where $O_k$ is the ring of integers of~$k$. Let $\pi \colon \mathcal{X} \to S$ be a stable curve of genus $g \geq 2$ with smooth generic fiber. The generic fiber of $\mathcal{X}$ can be viewed as a smooth projective geometrically connected curve of genus $g$ over $k$. Using arithmetic intersection theory one  associates to $\pi$ the absolute height of the Ceresa cycle $c(\mathcal{X}/S)$ as well as the absolute height of the canonical Gross-Schoen cycle $\langle \Delta, \Delta \rangle(\mathcal{X}/S)$. 

Both heights are elements of $\rr$. By  \cite[Theorem~1.5.6]{zhgs} the two heights are related by the simple formula
\begin{equation} \label{ceresa_gs_NF}
c(\mathcal{X}/S) = \frac{2}{3} \langle \Delta, \Delta \rangle (\mathcal{X}/S) \, .
\end{equation}
For each closed point $s \in S$ we denote by $\overline{G}_s$ the polarized weighted graph determined by the geometric fiber of $\pi$ at $s$, and for each complex embedding $v$ of $k$ we denote by $\lambda(\mathcal{X}_v)$ the archimedean $\lambda$-invariant of the smooth projective complex curve $\mathcal{X}_v = \mathcal{X} \otimes_v \cc$. Let $M_0$ denote the set of closed points of $S$, and write $M_\infty$ for the set of complex embeddings of $k$. We write
\begin{equation}
 \lambda(\mathcal{X}/S) =  \sum_{s \in M_0} \lambda(\overline{G}_s) \log Ns+ \sum_{v \in M_\infty} \lambda(\mathcal{X}_v) \, ,
\end{equation}
where $Ns$ denotes the cardinality of the residue field at~$s$. 

We write $h(\mathcal{X}/S) = \widehat{\deg} \det \pi_* \omega_{\mathcal{X}/S}$ for the absolute Faltings height of $\pi$. Here the line bundle $\det \pi_* \omega_{\mathcal{X}/S}$ is to be viewed as carrying the Hodge metric from \S\ref{sec:HRbeta} at the complex embeddings of $k$, and the degree $\widehat{\deg} $ denotes the Arakelov degree. 
A combination of (\ref{ceresa_gs_NF}) and \cite[Equation~(1.4.2)]{zhgs} leads to the fundamental relation
\begin{equation} \label{comparison_NF}
h(\mathcal{X}/S)  = \frac{3g-3}{2g+1} c(\mathcal{X}/S) +   \lambda(\mathcal{X}/S) 
\end{equation}
between the absolute Faltings height, the absolute height of the Ceresa cycle and the $\lambda$-invariants at the non-archimedean and archimedean places of $k$.

\subsection{Asymptotics of the $\lambda$-invariant}

In the comparison result (\ref{comparison_NF}) between the absolute Faltings height and the absolute height of the Ceresa cycle the combinatorial $\lambda$-invariant and the archimedean $\lambda$-invariant are placed on equal footing. Given this, our main technical result Theorem~\ref{lambda_asympt_intro} does not come as a surprise: in a one-parameter family of compact Riemann surfaces with stable degeneration, \emph{the asymptotics of the archimedean $\lambda$-invariant is essentially controlled by the combinatorial $\lambda$-invariant of the special fiber.}

In order to give the precise statement we first introduce some notation, to be used throughout the paper.

\medskip

\noindent \textit{Notation:} We write $\dd $ for  the open unit disk in $\cc$, and write $\dd^* = \dd \setminus \{0 \}$. 
 When $f ,g \colon \dd^* \to \rr$ are two continuous functions, we write $f \sim g$ if the difference $f-g$ extends to a continuous function over $\dd$.

\medskip

We consider a stable curve  $\pi \colon \mathcal{X} \to \dd$ of genus $g \geq 2$ over $\dd$. We assume that $\pi$ is smooth over $\dd^*$. Let $\overline{G} = (G, \mathbf{q})$ be the dual graph of the fiber $\mathcal{X}_0$ at the origin, viewed as a polarized weighted graph where the lengths of the edges are determined by the ``thicknesses'' of the singular points of $\mathcal{X}_0$ on the surface~$\mathcal{X}$. Let $\lambda(\overline{G})$ be the $\lambda$-invariant of $\overline{G}$ as discussed in \S\ref{sec:lambda_graph}.

\begin{thm} \label{lambda_asympt_intro} Write $X=\pi^{-1}\dd^*$ and let $\varOmega(t)$ be the family of period matrices on $\dd^*$ determined by a symplectic framing of $R^1\pi_*\zz_X$.
One has the asymptotics
\begin{equation} \label{lambda_estimate}
\lambda(X_t) \sim -\lambda(\overline{G})\log|t| -\frac{1}{2} \log \det \Im \varOmega(t)
\end{equation}
as $t \to 0$ over $\dd^*$. 
\end{thm}
It follows from the non-negativity of the slope (see \S\ref{sec:mori}) that we have $\lambda(\overline{G})>0$ when $G$ is not a point. Thus, if $\mathcal{X}_0$ is not smooth  we have that $- \lambda(\overline{G}) \log|t| $ is the main term of the asymptotic (\ref{lambda_estimate}). Indeed, if $b$ is the first Betti number of $G$ then we have $\det \Im \varOmega(t) \sim c \cdot (-\log|t|)^b$ for a suitable real number $c>0$. This follows from Remark~\ref{main_term}.

Theorem~\ref{lambda_asympt_intro} will be proved in \S\ref{curves}. Our proof is based on a general result on the asymptotics of the so-called $I$-invariant for degenerating principally polarized complex abelian varieties, see Theorem~\ref{I-asymptotic}. The latter result was recently also obtained independently by R.\ Wilms \cite[Theorem~1.1]{wi_deg}.

\subsection{Analytic expression for the height of the Ceresa cycle} 
We will use Theorem~\ref{lambda_asympt_intro}  to prove Theorem~\ref{main:intro} in \S\ref{lambda_asympt_gives_jump}.
As a second application of Theorem~\ref{lambda_asympt_intro} we have the following analytic formula for the height of the Ceresa cycle in the case of a function field over $\cc$. 

Let $\overline{S}$ be a smooth projective connected complex curve, let $D$ be an effective reduced divisor on $\overline{S}$, and write $S = \overline{S} \setminus D$. Let $\pi \colon \mathcal{X} \to \overline{S}$ be a stable curve of genus $g \geq 2$, and assume that $\pi$ is smooth over $S$. The pull-back of the Hain-Reed line bundle $\bb$ along the moduli map $S \to \mm_g$ yields a smooth hermitian line bundle $\bb_S$ on $S$. By the work of Brosnan and Pearlstein as discussed in \S\ref{sec:BP_work}, or alternatively by  \cite[Corollary~2.13]{bghdj},  we have that the first Chern form $c_1(\bb_S)$ extends as a semi-positive $(1,1)$-current $[c_1(\bb_S)]$ over $\overline{S}$.

As discussed in \S\ref{sec:height_Cer_FF} we have associated to the stable curve $\pi \colon \mathcal{X} \to \overline{S}$ the height $c(\mathcal{X}/\overline{S})$ of the Ceresa cycle of its generic fiber. We recall that the height $c(\mathcal{X}/\overline{S})$ is a non-negative rational number.
Our next result shows that $c(\mathcal{X}/\overline{S})$ has a simple expression in terms of the semi-positive $(1,1)$-current $[c_1(\bb_S)]$. 
\begin{thm} \label{height_Ceresa_intro} Let $\pi \colon \mathcal{X} \to \overline{S}$ be a stable curve of genus $g \geq 2$, and assume that $\pi$ is  smooth over $S$. 
The equality
\begin{equation} \label{eqn:height_Ceresa}
12(g-1) \, c(\mathcal{X}/\overline{S}) = \int_{\overline{S}} \,\, [c_1(\bb_S)]
\end{equation}
holds. In particular, the height $c(\mathcal{X}/\overline{S})$ of the Ceresa cycle vanishes if and only if $c_1(\bb_S) \equiv 0$ on $S$.
\end{thm}
One situation in which the height $c(\mathcal{X}/\overline{S})$ of the Ceresa cycle vanishes is when the generic fiber of $\pi$ is \emph{hyperelliptic}. 
In fact, the restriction of the Hain-Reed line bundle $\bb$ to the hyperelliptic locus in $\mm_g$ is trivial as a hermitian line bundle. 

We will derive Theorem~\ref{height_Ceresa_intro} from Theorem~\ref{lambda_asympt_intro} in \S\ref{sec:height_Ceresa}.

\subsection{Overview of the paper} 

In \S\ref{pm_graphs} we review some basic notions and  results on polarized graphs. Among other things, we introduce here the slope of a connected polarized weighted graph.
In \S\ref{vanishing_slope} we give a complete classification of those connected polarized graphs (of genus at least two) that have a vanishing slope. This classification leads to the characterization given in Theorem~\ref{cor:vanishing}. 

In \S\ref{trop_moment} we discuss some preliminary material concerning tropical moments. 
In \S\ref{deg_ab_var} we review some basic results on families of degenerating polarized complex abelian varieties. 
In \S\ref{sec:Iinv_asymp} we introduce the $I$-invariant of a principally polarized complex abelian variety following \cite{aut}  and prove Theorem~\ref{I-asymptotic} on the asymptotics of the $I$-invariant in  families of degenerating principally polarized complex abelian varieties. 
In \S\ref{curves} we deduce from Theorem~\ref{I-asymptotic} our main technical result Theorem~\ref{lambda_asympt_intro} on the asymptotics of the $\lambda$-invariant.

In \S\ref{sec:hodge} we review some basics on the Hodge metric and its asymptotics that we need in order to deduce Theorems~\ref{main:intro} and~\ref{height_Ceresa_intro} from Theorem~\ref{lambda_asympt_intro}. 
In \S\ref{sec:height_Ceresa} we give our proof of Theorem~\ref{height_Ceresa_intro}. 
In \S\ref{ht_jump} we review  the notion of height jumps as introduced by Hain in the setting of the Ceresa cycle, and as further analyzed by Brosnan and Pearlstein in a much wider setting. 
Finally in  \S\ref{lambda_asympt_gives_jump} we give our proof of  Theorem~\ref{main:intro}.

\renewcommand*{\thethm}{\arabic{section}.\arabic{thm}}

\section{Graphs and slopes} \label{pm_graphs}

\subsection{Graphs and polarizations}

In this paper, a \emph{graph} consists of a finite non-empty vertex set $V$ and a finite (possibly empty) edge set $E$, with the usual incidence relations. Loops and multiple edges are allowed. 

When $G=(V,E)$ is a graph with vertex set $V$ and edge set $E$, a \emph{polarization} of $G$ is a function $\mathbf{q} \colon V \to \zz_{\geq 0}$ with the property that for each $p \in V$ the inequality $v(p)-2+2\mathbf{q}(p) \geq 0$ is satisfied. Here $v(p)$ is the \emph{valency} of $p$, i.e., the number of emanating half-edges at $p$.  For $p \in V$ one usually calls $\mathbf{q}(p)$ the \emph{genus} of the vertex $p$. 
Further,  a \emph{divisor} on $G$ is an element of $\zz^V$. The \emph{degree} of a divisor $D = \sum_{p \in V} n_p \cdot p$ is the integer $\deg D = \sum_{p \in V} n_p$. A divisor $D = \sum_{p \in V} n_p \cdot p$ is called \emph{effective} if for all $p \in V$ the inequality $n_p \geq 0$ is satisfied.

Given a polarized graph $(G,\mathbf{q})$, the associated \emph{canonical divisor} is the effective divisor $K \in \zz^V$ given by setting $K(p) = v(p)-2+2\mathbf{q}(p)$ for $p \in V$. When $G$ is connected, the integer
\begin{equation} \label{eq:genus}
\overline{g}(G,\mathbf{q}) = \frac{1}{2}(\deg K +2 ) = b_1(G) + \sum_{p \in V} \mathbf{q}(p) 
\end{equation}
is called the \emph{genus} of $(G,\mathbf{q})$. Here $b_1(G) \in \zz_{\geq0}$ is the first Betti number of $G$. Note that the genus of a connected polarized graph is an element of $\zz_{>0}$.

\begin{definition}\label{def:stable}
A connected polarized graph $(G,\mathbf{q})$ is called {\em stable} if its canonical divisor is strictly positive.  
\end{definition}
It follows from the definition of the canonical divisor that, in a stable graph $(G,\mathbf{q})$, the polarization function $\mathbf{q}$ is strictly positive on all vertices of valency $1$ or $2$.  Further, it follows from \eqref{eq:genus} that a stable graph has genus at least $2$.

A {\em bridge} in a graph is an edge which, when removed, increases the number of connected components.
Let $G$ be a connected graph. An edge $e \in E$ is called \emph{of type $0$} if  it is not a bridge.
Let $\mathbf{q} \colon V \to \zz_{\geq 0}$ be a polarization of $G$ and let $e \in E$ be an edge. Write $g=\overline{g}(G,\mathbf{q})$.
 If the removal of a bridge $e $ from $G$ results in the disjoint union of two connected graphs, say $G_1$ and $G_2$ with polarizations $\mathbf{q}_1$ and $\mathbf{q}_2$, then it is easy to see that $g = \overline{g}(G_1,\mathbf{q}_1) + \overline{g}(G_2,\mathbf{q}_2)$. We call the bridge edge $e$ \emph{of type $h$}, where $h \in \{ 1, \ldots, [g/2] \}$, if one of the two components $G_1, G_2$ has genus~$h$ and the other has genus~$g-h$.

A \emph{length} on a graph $G=(V,E)$ is a map $m \colon E \to \rr_{>0}$. A graph equipped with a length function is called a \emph{weighted graph}. 
We will sometimes use the letter $G$ to refer to weighted graphs, if the edge lengths are given or clear from the context.

Let $\overline{G}=(G,\mathbf{q})$ be a connected polarized weighted graph with vertex set $V$ and edge set $E$. Let $g=\overline{g}(G,\mathbf{q})$ be its genus, and let $\delta(G)=\sum_{e \in E} m(e)$ denote the total length of the edges of $G$.  The total length of the edges of $G$ of type~$0$ is denoted by $\delta_0(G)$, and for $h \in \{ 1, \ldots, [g/2] \}$ the total length of the edges of $G$ of type~$h$ is denoted $\delta_h(\overline{G})$. We  have $\delta(G) = \delta_0(G)+ \sum_{h=1}^{[g/2]} \delta_h(\overline{G})$.

For the notion of \emph{(polarized) metric graph} that we shall employ we refer to either \cite[\S\S3 and~4]{ci}, \cite[Appendix]{zhadm}, or \cite[\S4]{zhgs}. A weighted graph $G$ naturally gives rise to a metric graph $\varGamma$ with a designated vertex set $V$, and a polarized weighted graph $\overline{G}=(G,\mathbf{q})$ naturally gives rise to a polarized metric graph $\overline{\varGamma}=(\varGamma,\mathbf{q})$, that is, a metric graph $\varGamma$ with designated vertex set $V$ and polarization $\mathbf{q} \colon V \to \zz_{\geq 0}$. 

The notions of canonical divisor and genus readily generalize to the setting of polarized metric graphs. Moreover, when $\overline{\varGamma}=(\varGamma,\mathbf{q})$ is a connected polarized metric graph, one naturally has its associated volume $\delta(\varGamma)$ as well as the invariants $\delta_0(\varGamma)$ and $\delta_h(\overline{\varGamma})$ for $h=1,\ldots,[g/2]$, where $g=\overline{g}(\overline{\varGamma})$ is the genus of $\overline{\varGamma}$.

\subsection{Minimal models}
Let $\overline{\varGamma} = (\varGamma,\mathbf{q})$ be a polarized metric graph. By a {\em model} of $\overline{\varGamma}$ we mean a polarized weighted graph $\overline{G}=(G,\mathbf{q})$ giving rise to $\overline{\varGamma}$. 

\begin{definition} \label{def:minmodel}
Let $\overline{\varGamma}=(\varGamma,\mathbf{q})$ be a connected polarized metric graph with canonical divisor $K$. Assume the genus of $\overline{\varGamma}$ is at least two.
The {\em minimal model} of $\overline{\varGamma}$ is the polarized weighted graph whose vertex set $V$ equals the support of $K$, and whose polarization is the restriction of $\mathbf{q}$ to $V$.
\end{definition}
In Definition~\ref{def:minmodel}, since the genus of $\overline{\varGamma}$ is assumed to be at least two, $V$ is non-empty. 
By definition $\mathbf{q}(p) = 0$ for all $p \in \varGamma \backslash V$ and, therefore, the minimal model has the same genus as $\overline{\varGamma}$.
Moreover, by construction, the minimal model of $\overline{\varGamma}$ is always a {\em stable} polarized graph (see Definition \ref{def:stable}).

\subsection{$2$-connected graphs and the ear decomposition}\label{sec:ear}

A metric graph is called {\em $2$-connected} if it is topologically $2$-connected; it cannot be disconnected by deleting a single point. A weighted graph is called {\em $2$-connected} if its associated metric graph is $2$-connected. Note that a weighted graph with two vertices connected by a single edge (i.e. an `edge segment') is not considered $2$-connected with our definition.

{\em Whitney's theorem} states that a graph $G$ is $2$-connected if and only if it has an {\em open ear decomposition} (see \cite[\S3.1]{Diestel}): there exists graphs $G_0, \ldots, G_k$ such that 
\begin{itemize}[leftmargin=*]
\item $G_0$ is a cycle in $G$, 
\item $G_k = G$, 
\item  for $1 \leq i \leq k$, the graph $G_i$ is obtained from $G_{i-1}$ by adding a path with both ends in $G_{i-1}$ and otherwise disjoint from $G_{i-1}$. 
\end{itemize}

\subsection{The $j$-function and effective resistance}
Let $\varGamma$ be a connected metric graph and fix two points $y,z \in \varGamma$. We recall that  $j_z^{\varGamma}(\cdot \, , y)$ denotes the unique continuous piecewise affine real valued function on $\varGamma$ satisfying:
\begin{itemize}
	\item[(i)] $\Delta \left(j_z^{\varGamma}(\cdot \, , y)\right )= \delta_y - \delta_z$, 
	\item[(ii)] $j_z^{\varGamma}(z,y) = 0$.
\end{itemize}
Here $\Delta$ is the Laplacian operator in the sense of distributions. 

The following result is well-known (see, e.g., \cite[Lemma 2.17]{bf_trop}).
\begin{lem} \label{lem:j_2conn}
A metric graph $\varGamma$ is $2$-connected if and only if $j_z^{\varGamma}(x, y) > 0$ for all $x,y,z \in \varGamma$ with $z \not\in \{x,y\}$.
\end{lem}

The {\em effective resistance} between two points $x,y \in \varGamma$ is defined as $r(x,y) = j^{\varGamma}_y(x,x)$. The effective resistance function $r \colon \varGamma \times \varGamma \to \rr$  is a distance function on $\varGamma$ (see, e.g., \cite[\S4.2]{djs_cross}).

The {\em Foster coefficient} of an edge segment $e$ with endpoints $e^+,e^- \in \varGamma$, is $\mathsf{F}(e) = 1-r(e^+,e^-)/m(e)$. It is well-known that $\mathsf{F}(e) \geq 0$. Moreover $\mathsf{F}(e) = 0$ if and only if $e$ is a bridge (see, e.g., \cite[\S7.4]{djs_cross}).

\subsection{Invariants of polarized graphs}

Let $\overline{\varGamma}=(\varGamma,\mathbf{q})$ be a connected polarized metric graph. We refer to \cite{zhadm} for the definition of the \emph{admissible measure} $\mu$ on $\varGamma$ associated to the polarization $\mathbf{q}$, and the \emph{admissible Green's function} $g_\mu \colon \varGamma \times \varGamma \to \rr$. 
We will be interested in the following invariants of $\overline{\varGamma}$, all introduced by Zhang \cites{zhgs, zhadm}. Let $K$ denote the canonical divisor of $\overline{\varGamma}$, and let $g=\overline{g}(\overline{\varGamma})$ denote its genus. 

\begin{itemize}[leftmargin=*]
\item The $\varphi$-invariant, given by
\begin{equation} \label{phi_nonarch} \varphi(\overline{\varGamma}) = -\frac{1}{4}\delta(\varGamma) + \frac{1}{4} \int_\varGamma g_\mu(x,x) \left((10g+2) \mu(x) - \delta_K(x)\right) \, .  
\end{equation}
\item The $\epsilon$-invariant, given by
\begin{equation} \label{eps_invariant}
\epsilon(\overline{\varGamma}) = \int_\varGamma g_\mu(x,x)\left((2g-2) \mu(x) + \delta_K(x)\right) \, .  
\end{equation} 
\item The $\lambda$-invariant, given by
\begin{equation} \label{lambda_nonarch}
\lambda(\overline{\varGamma}) =  \frac{g-1}{6(2g+1)} \varphi(\overline{\varGamma}) + \frac{1}{12}\left( \delta(\varGamma) + \epsilon(\overline{\varGamma})\right)\, . 
\end{equation}
\end{itemize}

\medskip

From the above definitions, one naturally has the invariants $\varphi(\overline{G})$, $\epsilon(\overline{G})$, and $\lambda(\overline{G})$ for polarized weighted graphs $\overline{G}=(G,\mathbf{q})$ as well. Note that these are invariants of the underlying polarized metric graph.

The following two examples follow easily from \cite[Proposition~4.9]{ci}.
\begin{example} \label{ex:loop} Assume that $\overline{\varGamma}$ is a loop graph of genus $g$ based on a single vertex. Then we have $(8g+4) \lambda(\overline{\varGamma}) = g \, \delta(\varGamma)$. 
\end{example}
\begin{example} \label{ex:segment} Assume that $\overline{\varGamma}$ is an edge segment  with endpoints of genera $h$ and $g-h$.  Then we have $(8g+4) \lambda(\overline{\varGamma}) = 4h(g-h) \, \delta(\varGamma)$. 
\end{example}

\subsection{Slopes of polarized graphs}
\begin{definition}
The \emph{slope} of a polarized metric graph $\overline{\varGamma}$ is given by
\begin{equation} \label{slope}
  s(\overline{\varGamma}) = (8g+4) \, \lambda(\overline{\varGamma}) - g \, \delta_0( \varGamma) - \sum_{h=1}^{[g/2]} 4h(g-h) \, \delta_h(\overline{\varGamma}) \, . 
 \end{equation}
\end{definition}
We also define the slope of a polarized weighted graph $\overline{G}=(G,\mathbf{q})$, denoted by $s(\overline{G})$, as the slope of the polarized metric graph modeled by $\overline{G}$. 

\begin{example} \label{ex:bridgelessSlopes}
Let $\overline{\varGamma} = (\varGamma,\mathbf{q})$ be a {\em bridgeless} connected polarized metric graph. Let $r \colon \varGamma \times \varGamma \to \rr$ denote its effective resistance function. Let $\overline{G}$ be any polarized weighted graph, with vertex set $V$ and edge set $E$, which is a model of $\overline{\varGamma}$.
For a vertex $p \in V$, we define
\[
\swt(p) = \sum_{q \in V}  r(p,q) \, \mathbf{q}(q)  
  + \sum_{e = \{e^+, \, e^-\} \in E} j^{\varGamma \setminus e}_p(e^+,e^-) \, \mathsf{F}(e) \, .
\]

From \cite[Proposition~4.6 and Proposition~4.15]{ci} we obtain an explicit formula for the slope of $\overline{\varGamma}$ as the $\swt$-weighted degree of the canonical divisor:
\[
s(\overline{\varGamma})   = \sum_{p\in V} \left(v(p)-2+2\mathbf{q}(p) \right) \swt(p) \, .
\]
It follows immediately that for all bridgeless polarized metric graphs $\overline{\varGamma}$ one has $s(\overline{\varGamma}) \geq 0$. 
\end{example}

\begin{example}  \label{lambda_for_two_gon} Let $\overline{\varGamma}$ be a polarized metric graph of genus $g$ consisting of two vertices of genera $h$ and $g-h-1$ and joined by two edges of lengths $m_1, m_2$.  Using Example \ref{ex:bridgelessSlopes}, we compute
\[  s(\overline{\varGamma}) = \frac{4m_1m_2}{m_1+m_2} (g-h-1)h   \, . \]
\end{example}
 
\subsection{Block-tree decomposition of graphs}\label{sec:block tree}
A {\em separation} of a connected graph $G$ is a decomposition into two connected subgraphs having a unique common vertex $v$ and disjoint nonempty edge sets. The common vertex $v$ is called a {\em separating vertex} of $G$.
A connected graph is called {\em inseparable} if it does not have a separating vertex.
A {\em block} of a connected graph $G$ is a maximal inseparable subgraph. If the edge set of $G$ is non-empty, then each block of $G$ is either a loop, a
bridge, or a maximal $2$-connected subgraph without loops. 

Every graph is a union of its blocks. Any two distinct blocks of a graph have at most one common vertex. 
Let $\mathcal{B}= \{H_i\}$ denote the collection of blocks of a connected graph $G$. Let $\mathcal{C}$ be the set of separating vertices of $G$. Define a bipartite graph $\mathcal{H}$ with the vertex set $\mathcal{B} \cup \mathcal{C}$ by connecting $H_i \in \mathcal{B}$ to $v_i \in \mathcal{C}$ if $v_i \in H_i$. Then $\mathcal{H}$ is a tree.
This is the famous {\em block-tree decomposition theorem} (see, e.g., \cite[\S3.1]{Diestel}).

Let $\overline{G}=(G,\mathbf{q})$ be a connected polarized weighted graph. Each block $H_i$ of $G$ is endowed with a natural structure of a polarized weighted graph $\overline{H_i}$, with the induced polarization $\mathbf{q}_i$ obtained from $\mathbf{q}$ by pushforward along the natural projection $G \to H_i$. With this induced polarization, each $\overline{H_i}$ has the same genus as $\overline{G}$.

\begin{lem}
Every block $H_i$ of a {\em stable} polarized graph $G$ is again a {\em stable} polarized graph.
\end{lem}
\begin{proof}
Consider the block decomposition of $G$ into its blocks $H_i$. Let $p$ be a vertex of $H_i$.

Assume $p$ is a non-separating vertex of $H_i$. The value of the canonical divisor of $H_i$ at $p$ is positive. This is because, in passing from $G$ to $H_i$, both the valency and the genus at $p$ are unchanged. 

Now assume $p$ is a separating vertex of $H_i$. The valency at $p$ in $H_i$ is at least~$1$ by connectedness. The value of the induced polarization $\mathbf{q}_i$ at $p$ is replaced by the genus of the entire subgraph of $G$ that touches the block $H_i$ at $p$. This genus is at least one. It follows that the value of the canonical divisor of $H_i$ at $p$ is again positive.
\end{proof}

The notions of blocks and block decomposition naturally extend to connected metric graphs. Let $\overline{\varGamma} = (\varGamma,\mathbf{q})$ be a connected polarized metric graph. By \cite[Theorem~4.3.2]{zhgs}, each of the invariants $ \varphi(\overline{\varGamma})$, $\epsilon(\overline{\varGamma})$ and $\lambda(\overline{\varGamma})$ is additive in the blocks of $\varGamma$. We conclude that also the slope $s(\overline{\varGamma}) $ is additive in the blocks of $\varGamma$. 

It follows from Example~\ref{ex:segment} that the slope of an edge segment is always zero. Therefore, if $\{\overline{\varGamma}_\alpha\}$ denotes the collection of 2-connected blocks of $\overline{\varGamma}$ (with their induced pushedforward polarizations), we have
\begin{equation} \label{additivity} 
s(\overline{\varGamma}) = \sum_\alpha s(\overline{\varGamma}_\alpha) \, .
\end{equation}

The following result follows immediately.
 \begin{prop} \label{prop:additivity} We have $s(\overline{\varGamma})= 0$ if and only if we have $s(\overline{\varGamma}_\alpha)=0$ for all $2$-connected blocks $\overline{\varGamma}_\alpha$ of $\overline{\varGamma}$.
 \end{prop}

\section{Vanishing slope} \label{vanishing_slope}

In this section, we give a complete classification of connected polarized metric graphs (of genus at least two) with vanishing slopes. Our strategy is to use the block-tree decomposition theorem, and reduce the classification problem to the study of blocks.

\begin{thm} \label{thm:slopezero_blocks}
Let $\overline{\varGamma} = (\varGamma,\mathbf{q})$ be a {\em bridgeless} connected polarized metric graph. Assume the genus of $\overline{\varGamma}$ is at least two and let $\overline{G} = (G,\mathbf{q})$ be its {\em minimal model}. Let $V$ be the vertex set of $\overline{G}$.
\begin{itemize}
\item[(a)] If $|V| = 1$ then $s(\overline{\varGamma}) = 0$.
\item[(b)] Assume $|V| \geq 2$. If $\mathbf{q} \not \equiv 0$ then $s(\overline{\varGamma}) >0$.
\item[(c)] Assume $G$ is $2$-connected. If $|V| \geq 3$ then $s(\overline{\varGamma}) >0$. 
\end{itemize}
\end{thm}

\begin{proof}
Since we are working with a bridgeless graph, we may use Example~\ref{ex:bridgelessSlopes} to compute $s(\overline{\varGamma})$ as the $\swt$-weighted degree of the canonical divisor: 
\[
s(\overline{\varGamma})   = \sum_{p\in V} \left(v(p)-2+2\mathbf{q}(p) \right) \swt(p) \, ,
\]
\[
\swt(p) = \sum_{q \in V}  r(p,q) \, \mathbf{q}(q)  + \sum_{e = \{e^+, \, e^-\} \in E} j^{\varGamma \setminus e}_p(e^+,e^-) \, \mathsf{F}(e) \, .
\]

By construction $\overline{G}$ is {\em stable}, so $v(p)-2+2\mathbf{q}(p)>0$ for all $p \in V$. Note also that all summands in the expression for $\swt(p)$ are non-negative. Therefore:
\begin{itemize}[leftmargin=*]
\item $s(\overline{\varGamma}) > 0$ if and only if $\swt(p) > 0$ for some $p\in V$.
\item $\swt(p) > 0$ if and only if at least one of the summands in its expression above is strictly positive.
\end{itemize}

\medskip

(a) Let $V= \{p\}$. Then $r(p,p) = 0$ and, for all edges $e$, we have $j^{\varGamma \setminus e}_p(e^+,e^-) = 0$ because $p=e^+=e^-$. It follows that $\swt(p) = 0$ and $s(\overline{\varGamma}) = 0$.

\medskip

(b) Let $q \in V$ be such that $\mathbf{q}(q)>0$. Consider another vertex $p \ne q$. Since the effective resistance function $r$ is a distance function, we must have $r(p,q) > 0$. It follows that $r(p,q) \mathbf{q}(q)>0$. We obtain $\swt(p) >0$ and, therefore, $s(\overline{\varGamma}) > 0$.

\medskip

(c) By  part (b) we may assume that $\mathbf{q} \equiv 0$. Since $v(p)-2+2\mathbf{q}(p) > 0$, for all $p \in V$, we must also have $v(p) \geq 3$. 

\medskip

\noindent {\bf Claim 1.} There exists an edge $e$ such that $G \backslash e$ remains $2$-connected. 

\medskip

\noindent {\em Proof of Claim 1.} Consider Whitney's construction of the $2$-connected graph $G$ via an open ear decomposition, described in \S\ref{sec:ear}. The last path added in the ear decomposition must be a single edge $e$. Otherwise, if it is a path of length at least two, the graph will have a vertex of valency $2$. Removing this last edge $e$ will result in a graph with an open ear decomposition, so the remaining graph will still be $2$-connected.

\medskip

\noindent {\bf Claim 2.} There exist $p \in V$ and $e = \{e^+,e^-\} \in E$ such that $j^{\varGamma \setminus e}_p(e^+,e^-) > 0$. 

\medskip

\noindent {\em Proof of Claim 2.} Let $e$ be as in Claim 1.  Let $p \not\in \{e^+,e^-\}$, which exists because $|V| \geq 3$. As $\varGamma \setminus e$ is $2$-connected, Lemma \ref{lem:j_2conn} guarantees $j^{\varGamma \setminus e}_p(e^+,e^-) > 0$.

\medskip

Since $G$ is bridgeless we have $F(e) > 0$ for any edge $e \in E$. Let $p \in V$ and $e \in E$ be as in Claim 2. Then  $j^{\varGamma \setminus e}_p(e^+,e^-) \, \mathsf{F}(e) > 0$. We obtain $\swt(p) >0$ and, therefore, $s(\overline{\varGamma}) > 0$.
\end{proof}

\begin{thm} \label{thm:slopezero}
Let $\overline{\varGamma} = (\varGamma,\mathbf{q})$ be a connected polarized metric graph. Assume the genus $g$ of $\overline{\varGamma}$ is at least two. Let $\overline{G} = (G,\mathbf{q})$ be its {\em minimal model}. The slope $s(\overline{\varGamma})$ vanishes if and only if one of the following two conditions holds:
\begin{itemize}
\item[(a)] $\mathbf{q} \equiv 0$ and $G$ is the graph on two vertices joined by $g+1$ parallel edges.
\item[(b)] The block-tree decomposition of $G$ consists of (at most $g$) loops and (an arbitrary number of) bridges. 
\end{itemize}
\end{thm}
Figure~\ref{fig:two graphs} illustrates the two possibilities. A graph of type (a) is sometimes called a {\em banana graph}. For a graph of type (b), after contracting all the bridge edges, we obtain a graph with one vertex and multiple loops. This is sometimes called a {\em bouquet of circles} or a {\em rose graph}. 

\begin{figure}[h!]
$$
\begin{xy}
(0,0)*+{
	\scalebox{.7}{$
	\begin{tikzpicture}
	\draw[black, ultra thick] (0,1.2) to [out=-170,in=90] (-2,0);
	\draw[black, ultra thick] (-2,0) to [out=-90,in=170] (0,-1.2);
	\draw[black, ultra thick] (0,1.2) to [out=-150,in=90] (-1,0);
	\draw[black, ultra thick] (-1,0) to [out=-90,in=150] (0,-1.2);
	\draw[black, ultra thick] (0,1.2) to [out=-10,in=90] (2,0);
	\draw[black, ultra thick] (2,0) to [out=-90,in=10] (0,-1.2);
	\draw[black, ultra thick] (0,1.2) to (0,-1.2);	
	\fill[black] (0,1.2) circle (.1);
	\fill[black] (0,-1.2) circle (.1);
	\end{tikzpicture}
	$}
};
(-16,0)*+{\mbox{{\tiny $m_1$}}};
(-9.1,0)*+{\mbox{{\tiny $m_2$}}};
(-2.1,0)*+{\mbox{{\tiny $m_3$}}};
(6,0)*+{\mbox{{ $\ldots$}}};
(18.1,0)*+{\mbox{{\tiny $m_{g+1}$}}};
(0,10.5)*+{\mbox{{\tiny $\color{gray} v_1$}}};
(0,-10.5)*+{\mbox{{\tiny $\color{gray} v_2$}}};
(0,-20)*+{(a)};
\end{xy}
\ \ \ \ \ \ \ \ \ \ 
\begin{xy}
(0,0)*+{
	\scalebox{.7}{$
	\begin{tikzpicture}
	\draw[black, ultra thick] (-3.5,0) to (3.2,0);
	\draw[black, ultra thick] (1,0) to (1,2.2);
	\fill[black] (-3.5,0) circle (.1);
	\fill[black] (-2.1,0) circle (.1);
	\fill[black] (-1,0) circle (.1);
	\fill[black] (1,0) circle (.1);
	\fill[black] (1,1.2) circle (.1);
	\fill[black] (1,2.2) circle (.1);
	\fill[black] (3.2,0) circle (.1);
	\draw[black, ultra thick] (-3.5,0) to [out=-40,in=0] (-3.5,-1.1);
	\draw[black, ultra thick] (-3.5,-1.1) to [out=180,in=-140] (-3.5,0);
	\draw[black, ultra thick] (-1,0) to [out=-40,in=0] (-1,-0.8);
	\draw[black, ultra thick] (-1,-0.8) to [out=180,in=-140] (-1,0);
	\draw[black, ultra thick] (-1,0) to [out=0,in=0] (-1,1.3);
	\draw[black, ultra thick] (-1,1.3) to [out=-180,in=180] (-1,0);
	\draw[black, ultra thick] (1,1.2) to [out=50,in=90] (1.8,1.2);
	\draw[black, ultra thick] (1.8,1.2) to [out=-90,in=-50] (1,1.2);
	\draw[black, ultra thick] (3.2,0) to [out=-60,in=0] (3.2,-1);
	\draw[black, ultra thick] (3.2,-1) to [out=180,in=-120] (3.2,0);
	\draw[black, ultra thick] (3.2,0) to [out=60,in=0] (3.2,0.8);
	\draw[black, ultra thick] (3.2,0.8) to [out=-180,in=120] (3.2,0);
	\draw[black, ultra thick] (3.2,0) to [out=60,in=90] (5,0);
	\draw[black, ultra thick] (5,0) to [out=-90,in=-60] (3.2,0);
	\end{tikzpicture}
	$}
};
(-28,-2)*+{\mbox{{\tiny $\color{gray} v_1$}}};
(-18,-2)*+{\mbox{{\tiny $\color{gray} v_2$}}};
(-10.5,-2)*+{\mbox{{\tiny $\color{gray} v_3$}}};
(4,-6.5)*+{\mbox{{\tiny $\color{gray} v_4$}}};
(1,4)*+{\mbox{{\tiny $\color{gray} v_5$}}};
(4,13.3)*+{\mbox{{\tiny $\color{gray} v_6$}}};
(21.5,-4.2)*+{\mbox{{\tiny $\color{gray} v_7$}}};
(-28,-13.2)*+{\mbox{{\tiny $m_1$}}};
(-23,-5.5)*+{\mbox{{\tiny $m_2$}}};
(-14.7,-5.5)*+{\mbox{{\tiny $m_3$}}};
(-10.5,-11)*+{\mbox{{\tiny $m_4$}}};
(-10.5,6)*+{\mbox{{\tiny $m_5$}}};
(-3.5,-5.5)*+{\mbox{{\tiny $m_6$}}};
(5.6,-0.5)*+{\mbox{{\tiny $m_7$}}};
(11.2,4)*+{\mbox{{\tiny $m_8$}}};
(5.6,8)*+{\mbox{{\tiny $m_9$}}};
(12,-5.5)*+{\mbox{{\tiny $m_{10}$}}};
(19,-12.5)*+{\mbox{{\tiny $m_{11}$}}};
(34.2,-4)*+{\mbox{{\tiny $m_{12}$}}};
(19,3)*+{\mbox{{\tiny $m_{13}$}}};
(0,-20)*+{(b)};
\end{xy}
$$
\captionsetup{singlelinecheck=off}
\caption[]{
Polarized metric graphs with vanishing slopes, represented by their minimal models (Theorem~\ref{thm:slopezero}). Vertices of the minimal model are denoted by $v_i$ and the edge lengths are denoted by $m_i$.
\begin{itemize}
\item[(a)] The graph with $g+1$ parallel edges and $\mathbf{q} \equiv 0$.
\item[(b)] An example of a graph whose block-tree decomposition consists of only loops and bridges. The sum of $\mathbf{q}(v_i)$'s plus the number of loops is equal to $g$.
\end{itemize}}
\label{fig:two graphs}
\end{figure}

\begin{proof}
Assume $s(\overline{\varGamma}) = 0$. By Proposition~\ref{prop:additivity}, for each $2$-connected block $H_i$ of $G$, we must have $s(\overline{H_i}) = 0$. By Theorem~\ref{thm:slopezero_blocks} part (c), $H_i$ cannot have more than two vertices. In the case that $H_i$ has precisely two vertices, by Theorem~\ref{thm:slopezero_blocks} part (b), we must have $\mathbf{q} \equiv 0$.

\medskip

\noindent {\bf Claim.} If $G$ has a block $H_i$ with precisely two vertices, then $G = H_i$.

\medskip

\noindent {\em Proof of the Claim.} Assume $G$ has more than one block. Let $p$ be a separating vertex of $G$ in $H_i$. 
The value of the induced polarization $\mathbf{q}_i$ at $p$ is replaced by the genus of the entire subgraph of $G$ that touches the block $H_i$ at $p$. This genus is at least one. By Theorem~\ref{thm:slopezero_blocks} part (b) we obtain $s(\overline{H_i}) > 0$, a contradiction.

\medskip

The `only if' part now follows from the block-tree decomposition theorem, described in \S\ref{sec:block tree}. 
The `if' part is an immediate consequence of Proposition~\ref{prop:additivity} and the explicit formula in Example~\ref{ex:bridgelessSlopes}.
\end{proof}
We note that Theorem~\ref{thm:slopezero}, together with Theorem~\ref{main:intro}, proves Theorem~\ref{cor:vanishing}.

\section{Tropical moments}  \label{trop_moment}

The purpose of this section is to define the tropical moment of a lattice, of a positive definite matrix, and of a principally polarized tropical abelian variety. This section serves as a preparation for the results and proofs given in \S\ref{sec:Iinv_asymp}.

\subsection{Tropical moment of a lattice}
A \emph{lattice} is the datum of a finitely generated free abelian group $\varLambda$ together with an inner product $[ \cdot,\cdot]$ on the real vector space $\varLambda \otimes \rr$. Let $(\varLambda,[\cdot,\cdot])$ be a lattice and denote by $\|\cdot\|$ the resulting norm on the vector space $V=\varLambda \otimes \rr$. We define the \emph{Voronoi region} associated to $(\varLambda,[\cdot,\cdot])$ to be the symmetric, convex, compact subset
\[ \Vor(\varLambda) = \{ v \in V \, | \, \forall{\lambda \in \varLambda} \, : \|v\| \leq \|v - \lambda \| \}  \]
of $V$. The compact set $\Vor(\varLambda)$ is a fundamental domain of the lattice $\varLambda$ in $V$, in the sense that the natural map $\Vor(\varLambda) \to V/\varLambda$ is surjective, and injective on the interior of $\Vor(\varLambda)$.

Let $\mu_L$ be a Lebesgue measure on $V$. We define the \emph{tropical moment} of the lattice $(\varLambda,[\cdot,\cdot])$ to be the quotient
\begin{equation} \label{def-I}
 I(\varLambda)  = \frac{\int_{\Vor(\varLambda)} \| v \|^2 \, \d \, \mu_L  }{ \int_{\Vor(\varLambda)} \d \, \mu_L } \, .
\end{equation}
The real number $I(\varLambda)$ is independent of the choice of Lebesgue measure $\mu_L$. 

\subsection{Tropical moment of a positive definite matrix} 

Let $g \in \zz_{\geq 0}$ and let $Z \in \Mat(g \times g,\rr)$ be a positive definite $g\times g$ matrix. Then we denote by $I(Z)$ the tropical moment of the lattice $ \zz^g$ in $\rr^g$ where the inner product on $\rr^g$ is given by $(\alpha,\beta) \mapsto \alpha^t Z   \beta$, i.e.\ the inner product on $\rr^g$ that has Gram matrix $Z$ on the standard basis. We refer to $I(Z)$ as the tropical moment of the positive definite matrix $Z$. It is easy to see that $I(Z)$ is the tropical moment of any lattice for which $Z$ is a Gram matrix on a basis. 

We have the following explicit formula for $I(Z)$. We write $\Vor(Z)$ for the Voronoi region of the lattice $(\zz^g,Z)$. The standard Lebesgue measure on $\rr^g$ gives $\Vor(Z)$ volume one. Indeed, the unit box $[0,1]^g$ has volume one, both the unit box and $\Vor(Z)$ are fundamental domains for $\zz^g$ in $\rr^g$, and all fundamental domains for $\zz^g$ in $\rr^g$ have the same volume. Formula (\ref{def-I}) thus specializes to give
\begin{equation} \label{formula-I}
 I(Z) = \int_{\beta \in \Vor(Z)} \beta^t  Z   \beta \, \d \, \beta \, . 
\end{equation}
We note that for $\beta \in \rr^g$, the condition that $\beta \in \Vor(Z)$ can be written as
\begin{equation} \label{when_in_Vor}
 \forall \, n \in \zz^g \, : \, n^t    Z   \beta \leq \frac{1}{2} n^t    Z   n \, . 
\end{equation}
It immediately follows that for all $\lambda \in \rr_{>0}$ we have $\Vor(\lambda Z) = \Vor(Z)$. By (\ref{formula-I}) we then find $I(\lambda Z) = \lambda I(Z)$.

We will also work with the following slight generalization.
 Let $V \subset \rr^g$ be any compact set, and let $Z \in  \Mat(g \times g, \rr)$ be any $g\times g$ matrix with entries in $\rr$. Then we define
\begin{equation} \label{def:IV}
 I_V(Z) = \int_{\beta \in V} \beta^t  Z   \beta \, \d \, \beta \, . 
\end{equation}
Thus, in particular, when $Z$ is positive definite we have $I(Z) = I_{\Vor(Z)}(Z)$. We can view $I_V$ as a linear functional on $\Mat(g \times g, \rr)$.

The proof of the following lemma is left to the reader.
\begin{lem} \label{IVA_IVQ0} Let $g \in \zz_{\geq 0}$ and let $r \in \{0,\ldots, g\}$. Let $A$ be a real $g\times g$ matrix in block form
\[ A = \begin{pmatrix}[c|c] A_0 & 0 \\
\hline 
0 & 0 \\
\end{pmatrix}  \, , \]
with $A_0$ positive definite and of size $r\times r$.   Let $Z$ be a real $g\times g$ matrix in block form
\[ Z = \begin{pmatrix}[c|c] Z_{r,r} & Z_{r,g-r} \\
\hline 
Z_{g-r,r} & Z_{g-r,g-r} \\
\end{pmatrix}  \, , \]
with $Z_{g-r,g-r}$ positive definite and of size $(g-r)\times(g-r)$. Write 
\[ V = \Vor(A_0) \times \Vor(Z_{g-r,g-r}) \subset \rr^r \times \rr^{g-r} = \rr^g \, . \]
Then  $V$ is a fundamental domain for $\zz^g$ in $\rr^g$, and the equalities  
\[ I_V(A) = I(A_0) \]
and
\[ I_V(Z ) = I_{\Vor(A_0)}(Z_{r,r} ) + I(Z_{g-r,g-r} )   \]
hold.
\end{lem}

\subsection{Tropical moment of a  principally polarized tropical abelian variety} \label{sec:pptav}

For the purposes of the present paper, a \emph{principally polarized tropical abelian variety} consists of two finitely generated free abelian groups $X$, $Y$, together with an isomorphism $\varPhi \colon Y \isom X$ and a homomorphism $b \colon Y \to X^*=\Hom(X,\zz)$, with the property that the composite $\varPhi^* \circ b \colon Y \to Y^*$ is a non-degenerate self-dual map that defines an inner product $[\cdot,\cdot]_Y$ on~$Y_\rr$. 

Let $(X, Y, \varPhi, b)$ be a principally polarized tropical abelian variety. Then we define the tropical moment of $(X,Y,\varPhi,b)$ to be the tropical moment of the lattice $Y$ inside the inner product space $(Y_\rr,[\cdot,\cdot]_Y)$.

Let $\varSigma = Y_\rr/Y$; then we view $\varSigma$ as a polarized real torus, with inner product on its tangent space $Y_\rr$ given by $[\cdot,\cdot]_{Y}$. We will often write $I(\varSigma)$ for the tropical moment of the principally polarized tropical abelian variety $(X, Y, \varPhi, b)$.

\section{Degenerating abelian varieties}  \label{sec:AbVarHS} \label{deg_ab_var}

The purpose of this section is to briefly review some of the structures underlying families of degenerating principally polarized complex abelian varieties. The material in this section expands on the discussion in \cite[Chapter~II.0]{fc}.

\subsection{Normal form} \label{sec:standard_form}
Let $g \in \zz_{\geq 0}$. For $m_1, m_2, n_1, n_2 \in \zz^g$ we set
\begin{equation} \label{polar} E((m_1,m_2),(n_1,n_2 )) = m_1^t n_2 - m_2^t n_1 \,  .
\end{equation}
Thus $E$ defines the standard symplectic form on the free abelian group $\zz^{2g}$. We write
\begin{equation}
 \mathbb{H}_g = \{ \varOmega \in \Mat(g \times g,\cc) \, | \, \varOmega = \varOmega^t \,\, \textrm{and} \,\, \Im \varOmega > 0 \} 
\end{equation}
 for the Siegel  upper half space in degree~$g$. 
For each $\varOmega  \in \mathbb{H}_g$ we have a lattice $\varLambda_\varOmega = \zz^g + \varOmega \zz^g$ inside $\cc^g$ and we write $A_\varOmega$ for the complex torus $\cc^g / \varLambda_\varOmega$. 
The complex torus $A_\varOmega$ has a natural structure of principally polarized complex abelian variety, where the associated symplectic form $\varLambda_\varOmega \times \varLambda_\varOmega \to \zz$ is given by sending $((m_1+ \varOmega m_2), (n_1 + \varOmega n_2))$ to $E((m_1,m_2),(n_1,n_2 ))$. 

\subsection{An equivalence of categories} \label{sec:equiv_cat}
Conversely, every  principally polarized complex abelian variety $(A, a \colon A \isom A^t)$ is isomorphic to one of the form $A_\varOmega$.
We can make this statement more precise using the language of Hodge structures.
Namely, a polarized complex  abelian variety $(A,a \colon A \to A^t)$ can equivalently be thought of as a polarized Hodge structure of type $(-1,0)$, $(0,-1)$. 

Given a polarized complex abelian variety $(A,a)$ the associated polarized Hodge structure is the free abelian group $H= H_1(A,\zz)$ with $F^{0}H_\cc = H^0(A,\varOmega^1)^\lor$ and with polarization $H_1(A,\zz) \to H_1(A,\zz)^*(1) = H_1(A^t,\zz) $ given by $a_*$. Vice versa, given a polarized Hodge structure $(H,F^\bullet,Q)$ of type $(-1,0)$, $(0,-1)$ the associated polarized complex abelian variety has underlying torus given by the intermediate Jacobian $JH = H_\cc / (F^0H_\cc + H)$ and polarization given by $Q$. 

\subsection{Period matrices}
When the polarization $Q \colon H \to H^*(1)$ is an isomorphism, upon choosing a symplectic basis of $H$, the pair $(H,Q)$ can be identified with the standard symplectic space $\zz^{2g} $ with polarization given by the form $E$ as in (\ref{polar}). Under this identification,
the $g$-dimensional complex subspace $F^0H_\cc \subset H_\cc $ corresponds to the row space in $\cc^{2g}$ of the matrix $( \mathrm{Id}_g \, | \, \varOmega )$ for a uniquely determined $\varOmega \in \mathbb{H}_g$. 

We call $\varOmega$ the \emph{period matrix} of $(H,F^\bullet,Q)$ on the chosen symplectic basis. We see that the moduli space $\aa_g$ of principally polarized complex abelian varieties of dimension~$g$ (viewed as an orbifold) is naturally identified with $\Sp(2g,\zz) \setminus \mathbb{H}_g$. The map $\cc^{2g} \to \cc^g$ given by $(m_1,m_2) \mapsto -\varOmega m_1 + m_2$ for $m_1, m_2 \in \cc^g$ induces an isomorphism of the principally polarized abelian variety $JH =H_\cc/(F^0H_\cc + H)$ with the principally polarized abelian variety $A_\varOmega = \cc^g / (\zz^g + \varOmega \zz^g)$.

\subsection{The period map} \label{sec:period_map}
Let $S$ be a complex manifold, and let $(f \colon \mathcal{A} \to S, \alpha)$ be a holomorphic family of polarized complex abelian varieties. This datum can equivalently be thought of as a polarized variation of pure Hodge structures $\mathcal{H}$ of type $(-1,0)$, $(0,-1)$ over $S$. Globalizing the construction from \S\ref{sec:equiv_cat}, the equivalence is given by setting $\mathcal{H} = R^1f_* \zz_\mathcal{A}(1)$ and endowing the vector bundle $\mathcal{H}_\cc \otimes \oo_S$ with the natural Hodge filtration determined by $\mathcal{F}^0 \left(\mathcal{H}_\cc \otimes \oo_S \right)= (f_* \varOmega^1_{\mathcal{A}/S})^\lor$. 

Assume that $S$ is connected and let  $H$ denote the fiber of $\mathcal{H}$ at a chosen basepoint in $S$. The family $\alpha$ of polarizations translates into a polarization $Q \colon H \to H^*(1)$. We will assume throughout that the polarizations $\alpha$ are principal; equivalently, we assume that the map $Q$ is an isomorphism. Let $\widetilde{S} \to S$ be a universal covering space. The variation $\mathcal{H}$ pulls back to $\widetilde{S}$ to give a trivial local system $\widetilde{\mathcal{H}}$ with fiber $H$ and polarization $Q \colon H \isom H^*(1)$. As is customary to do, we choose a symplectic basis of $H$ and view it as a global basis of the pullback variation $\widetilde{\mathcal{H}}$ over $\widetilde{S}$. By taking the associated period matrices as in \S\ref{sec:equiv_cat} we obtain a holomorphic \emph{period map} $ \varOmega\colon \widetilde{S} \to \mathbb{H}_g$ that we often think of as a multi-valued holomorphic map $\varOmega \colon S \to \mathbb{H}_g$.

In the next sections we assume that the base manifold $S$ is the punctured unit disk $\dd^*$ in $\cc$. We choose an identification $\widetilde{S} = \mathbb{H}$ with universal covering map $\widetilde{S} \to S$ given by $u \mapsto \exp(2\pi i u)$. We shall assume throughout that the variation $\mathcal{H}$ has unipotent monodromy around the origin of $\dd$. This means that  the  family $f \colon \mathcal{A} \to \dd^*$ extends to a holomorphic family of semiabelian varieties $\mathcal{G} \to \dd$. 

\subsection{Associated semiabelian variety}
We denote by $G=\mathcal{G}_0$ the complex semiabelian variety that results  from taking the fiber of $\mathcal{G}$ at the origin. We then have a canonical short exact sequence of complex group varieties
\begin{equation} \label{eqn:ses}
 1 \to T \to G \to P \to 0 \, , 
\end{equation}
with $T  $ the toric part of $G$, and $P=G/T$ an abelian variety. We write  $X=\Hom(T,\gg_\mathrm{m})$ for the character group of $T$ and $X^*=\Hom(X,\zz)=H_1(T)$ for its co-character group. 

The family $f^t \colon \aa^t \to \dd^*$ of dual abelian varieties  similarly gives rise to a complex semiabelian variety $G^t$, fitting in a short exact sequence
\[ 1 \to T^t \to G^t \to P^t \to 0 \, , \]
with $T^t$ the dual torus of $T$ and $P^t=G^t/T^t$ the dual abelian variety of $P$. The holomorphic family $\alpha$ of principal polarizations naturally gives rise to an isomorphism of short exact sequences
\begin{equation} \label{morphism_ses} \xymatrix{ 1 \ar[r] & T  \ar[d]^{\cong} \ar[r] & G \ar[d]^{\cong} \ar[r] & P \ar[d]^{\cong} \ar[r] & 0  \\ 
1 \ar[r] & T^t \ar[r] & G^t \ar[r] & P^t \ar[r] & 0 \, . }   
\end{equation}
We write $Y$ for the character group of $T^t$. Pulling back along the isomorphism $T \isom T^t$ in (\ref{morphism_ses}) we find an isomorphism $\varPhi \colon Y \isom X$ of abelian groups.

\subsection{Associated limit mixed Hodge structure}
The fiber $H$ is equipped with a natural limit mixed Hodge structure. Its weight filtration can be defined over $\zz$, and takes the form
\[ 0 \subseteq W_{-2}H = X^* \subseteq W_{-1}H \subseteq W_0H = H \, . \]
One may view $H^*(1)$ as the limit mixed Hodge structure of the family of dual abelian varieties. Thus similarly  $H^*(1)$ has a natural limit mixed Hodge structure, with weight filtration
\[ 0 \subseteq W_{-2}H^*(1) = Y^* \subseteq W_{-1}H^*(1) \subseteq W_0H^*(1) = H^*(1) \, . \]
We have identifications $W_{-1}H = H_1(G)$ and $W_{-1}H^*(1)=H_1(G^t)$. Also we have  identifications of the weight graded pieces
\[ \Gr_{-2}H = X^* \, , \quad \Gr_{-1}H = H_1(P) \, , \quad \Gr_0 H = Y \, , \]
where the identity $\Gr_0 H = Y $ can be seen from the fact that the weight filtration of the dual $H^*(1)$ is  identified with the dual of the weight filtration of $H$. 

The limit mixed Hodge structure $H$ has a graded polarization induced from the polarization $Q$. For instance, the pure Hodge structure $\Gr_{-1}H$ corresponds to the abelian variety $P$; the natural induced polarization $P \to P^t$ coincides with the isomorphism given in (\ref{morphism_ses}) and is thus principal. Also we obtain a commutative diagram
\begin{equation} \label{comm_diagrams} \xymatrix{ H \ar[r]^{Q}_{\cong \quad}    & H^*(1)    \\
X^* \ar[u] \ar[r]^{\varPhi^*}_{\cong \quad} & Y^* \ar[u] 
}   
\end{equation}
where the vertical arrows are the natural inclusions. 

\subsection{Monodromy pairing} \label{sec:monodr_pairing}
Let $T \colon H \isom H$ denote the monodromy operator determined by a positively oriented loop in $\pi_1(\dd^*) \cong \zz$, and write $N = T - \id_H$. The map $N \colon H \to H$ is  nilpotent and descends to give a natural map 
\[b \colon Y = \Gr_0 H \longrightarrow \Gr_{-2}H = X^* \, , \] 
called the \emph{monodromy pairing}. The free abelian groups $X$, $Y$ together with the homomorphisms $\varPhi \colon Y \isom X$ and $b \colon Y \to X^*$ form a principally polarized tropical abelian variety (in the sense of \S\ref{sec:pptav}) canonically associated to the holomorphic family $(f \colon \mathcal{A} \to \dd^*, \alpha)$ of principally polarized complex abelian varieties. 

We denote by $[\cdot,\cdot]_Y$ the  inner product on $Y_\rr$ coming from the homomorphism $\varPhi^* \circ b \colon Y \to Y^*$. 
Let $\psi \colon H \to Y=\Gr_0 H$ denote the natural projection map. Via the commutative diagram (\ref{comm_diagrams}) and the relation $N = b \circ \psi$ we arrive at the equalities
\begin{equation} \label{monodromy}
\begin{split} [\psi(h),\psi(k)]_{Y} & =  (\varPhi^* \circ b )(\psi(h))(\psi(k)) \\
 & = \varPhi^*(Nh) (\psi(k)) \\
 & = Q(Nh,k) 
 \end{split}
\end{equation}
for $h ,k \in H$.

Let $r = \dim T$ be the toric rank of the semi-abelian variety~$G$.  We choose a symplectic basis of $H$ and view it as a global basis of the pullback variation $\widetilde{\mathcal{H}}$ over the universal covering space $\mathbb{H}$ of $\dd^*$ as in \S\ref{sec:period_map}. The symplectic basis may be chosen in such a way that the map $N \colon H \to H$ is represented by a matrix
\begin{equation} \label{mono_matrix}
 M = \begin{pmatrix}[c|c] 0 & A \\
\hline 
0 & 0 \\
\end{pmatrix}  \, , 
\end{equation}
where $A$ is a symmetric integral $g\times g$ matrix in block form
\[ A = \begin{pmatrix}[c|c] A_0 & 0 \\
\hline 
0 & 0 \\
\end{pmatrix}  \, , \]
with $A_0$ positive definite and of size $r\times r$. 
\begin{prop} \label{monodromy-Gram-matrix} The matrix $A_0$ is a Gram matrix for the lattice $Y$ in the inner product space $(Y_\rr,[\cdot,\cdot]_Y)$.
\end{prop}
\begin{proof}  
 Let $h_1, h_2, k_1, k_2 \in \zz^g$.   Denote by $h$, $k$ the elements of $H$ corresponding to $(h_1,h_2)$, $(k_1,k_2) \in \zz^{2g}$. For $m \in \zz^g$ we denote by $m^{(r)} \in \zz^r$ the first $r$ coordinates of $m$. Using (\ref{polar}), (\ref{monodromy}) and (\ref{mono_matrix}) we compute
\begin{equation} \label{Gram}
 \begin{split} [\psi(h),\psi(k) ]_{Y} & = Q(Nh,k) \\
& = E(M (h_1,h_2),(k_1,k_2)) \\
& = E((A h_2,0),(k_1,k_2)  ) \\
& = h_2^t A k_2 \\
& = h_2^{(r),t} A_0 k_2^{(r)}  \, . 
\end{split} 
\end{equation}
Let $S \subset H$ be the subset corresponding  to the set of elements $(0,k_2) \in \zz^{2g} $ with $k_2 \in \zz^g$ of the form $(m_2,0)$ with $m_2$ a standard basis vector in $\zz^r$.
By (\ref{mono_matrix}) a basis of the free abelian group $NH \subset X^*$ is given by the elements $Nk$ where $k$ runs through $S$. Since $NH=b(Y)$, and $b$ is injective, we conclude that  a basis of the free abelian group $Y =\Gr_0 H $ is given by the elements $\psi(k)$ with $k$ running through $S$. The equalities in (\ref{Gram}) then yield the proposition.
\end{proof}

\subsection{Nilpotent Orbit Theorem} \label{sec:NilpOrb}
As in \S\ref{sec:period_map} we let $\varOmega \colon \mathbb{H} \to \mathbb{H}_g$ denote the period map associated to the family $(f \colon \mathcal{A} \to \dd^*, \alpha)$ and our choice of symplectic basis turning the nilpotent map $N$ into the form (\ref{mono_matrix}). 

The map $\mathbb{H} \to \Mat(g \times g,\cc)$ given by $u \mapsto -A u + \varOmega(u)$ descends via the universal covering map $\mathbb{H} \to \dd^*$ to give a map $\dd^* \to  \Mat(g \times g,\cc)$. By the Nilpotent Orbit Theorem \cite[Theorem~4.9]{schmid} this map extends to a holomorphic map $B \colon \dd \to  \Mat(g \times g,\cc)$. We conclude that the period map $\varOmega \colon \mathbb{H} \to   \mathbb{H}_g$  can be written as the multi-valued map
\begin{equation} \varOmega(t) = \frac{1}{2\pi i} A \log t + B(t)  \, , \quad t \in \dd^* \, ,
\end{equation}
with $A \in \Mat(g \times g,\zz)$ and $B \colon \dd \to  \Mat(g \times g,\cc)$ as above. 

For every $t \in \dd$ we denote by $B_{g-r,g-r}(t)$   the lower right $(g-r)\times(g-r)$ block of the matrix $B(t)$. We have the following: 
\begin{itemize}
\item[(a)] after possibly shrinking $\dd$ the map $B_{g-r,g-r} \colon \dd \to \Mat((g-r) \times (g-r),\cc)$ factors through the Siegel upper half space $\mathbb{H}_{g-r}$; 
\item[(b)] the matrix $B_{g-r,g-r}(0) \in \mathbb{H}_{g-r}$ is a period matrix of the principally polarized complex abelian variety $P=G/T$ that appears in (\ref{eqn:ses}).
\end{itemize}

\section{The $I$-invariant and its asymptotics} \label{sec:Iinv_asymp}

As was announced in the introduction, our proof of Theorem~\ref{lambda_asympt_intro} is based on an analysis of the asymptotics of the so-called $I$-invariant of principally polarized complex abelian varieties, introduced by P.\ Autissier in \cite{aut}. In the first section below we introduce the $I$-invariant, and we analyze its asymptotics in the follow-up sections.

\subsection{The $I$-invariant} \label{sec:Iinv}
 Let $A$ be a complex abelian variety, endowed with a principal polarization $a \colon A \isom A^t$.  Let $L$ be a symmetric ample line bundle on $A$ determining the given polarization and let $s$ be a non-zero global section of $L$. Equip $L$ with a cubical metric $\|\cdot\|$ (i.e., a smooth metric whose curvature form is translation-invariant) and let $\mu_H$ denote the Haar measure on~$A$, normalized to give $A$ volume one. The $I$-invariant of $(A,a)$  is then defined to be the real number
\begin{equation} \label{def:Iinv}
 I(A,a) = -\int_{A} \log \|s\| \, \d \, \mu_H + \frac{1}{2} \log \int_{A} \|s\|^2 \, \d \, \mu_H \, .
\end{equation}
It can be verified that the real number $I(A,a)$ is independent of the choice of the symmetric ample line bundle $L$, the global section $s$ and the cubical metric $\| \cdot \|$. The Jensen inequality implies that $I(A,a) > 0$.

The $I$-invariant (\ref{def:Iinv}) can be written more explicitly using Riemann's theta function. Set $g=\dim(A)$. We recall from \S\ref{sec:standard_form}  that we can think of $(A,a)$ as  the principally polarized complex abelian variety $A_\varOmega = \cc^g/(\zz^g + \varOmega \zz^g)$ for a suitable $\varOmega \in \mathbb{H}_g$. The Riemann theta function is the function
\begin{equation} \label{def:Riemann_theta}
 \theta(z,\varOmega) = \sum_{n \in \zz^g} \exp(\pi i n^t \varOmega n + 2\pi i n^t z) \, , \quad z \in \cc^g \,  . 
\end{equation}
We shall be working with the following normalized version (cf.\ \cite[p.~401]{fa})
\begin{equation} \label{norm_theta}
 \|\theta\|(z,\varOmega) = (\det \Im \varOmega)^{1/4} \exp(-\pi (\Im z)^t (\Im \varOmega)^{-1} (\Im z) ) | \theta(z,\varOmega) |\, .
 \end{equation}
It can be checked that the function $\|\theta\|(z,\varOmega)$ descends to the complex torus $A_\varOmega$. In fact we can view the function $\|\theta\|$ as giving the norm, in a suitable cubical metric, of the standard global section~$1$ of the standard symmetric ample line bundle $L=\oo(\divisor \theta)$ associated to the principal polarization of $A_\varOmega$.

 Let $\mu_H$ denote the Haar measure on $A_\varOmega$, normalized to give $A_\varOmega$ unit volume. Then by \cite[Proposition~8.5.6]{bila} we have
\begin{equation} \label{integral_constant}
 \int_{A_\varOmega} \|\theta\|^2 \, \d \, \mu_H = 2^{-g/2} \, . 
\end{equation}
From this we arrive at the explicit formula
\begin{equation} \label{def_arch_lambda_bis}
I(A_\varOmega) = -\int_{A_\varOmega} \log \|\theta\|  \, \d \, \mu_H - \frac{g}{4} \log 2 
\end{equation}
for the $I$-invariant of the principally polarized complex abelian variety $A_\varOmega$.

\subsection{Asymptotics of the $I$-invariant}

 As is suggested by \cite[Proposition~4.1]{aut},  the invariant $I \colon \mathcal{A}_g \to \rr_{>0}$ may be viewed as the minus logarithm of a distance to the boundary of the moduli space of principally polarized abelian varieties. The purpose of this section is  to make this idea more precise. 
 
 In fact, we determine the asymptotics of the $I$-invariant in arbitrary one-parameter degenerations with unipotent monodromy, see Theorem~\ref{I-asymptotic}. This result was recently also obtained independently by R.~Wilms \cite[Theorem~1.1]{wi_deg}. 
Our proof is slightly different from Wilms's and yields, as a by-product, a simple expression for the implied limiting value at zero -- see  \S\ref{sec:cst_term}.
 
Let $(f \colon \mathcal{A} \to \dd^*, \alpha)$ be a family of principally polarized complex  abelian varieties with unipotent monodromy around the origin. We recall that this is equivalent to saying that the family $f$ has semiabelian reduction over $\dd$. We let $\varSigma$ denote the  polarized real torus determined by the principally polarized tropical abelian variety associated to~$f$ as in \S\ref{sec:monodr_pairing}, and write  $I(\varSigma)$ for the tropical moment of $\varSigma$ as in \S\ref{sec:pptav}.

\begin{thm} \label{I-asymptotic}   Let $\varOmega(t)$ be the family of period matrices on $\dd^*$ determined by a symplectic basis of a fiber of $R^1f_* \zz_\mathcal{A}$. The asymptotics
\begin{equation} \label{eqn:I-asymp}
2\,  I(\mathcal{A}_t,\alpha_t) \sim  -I(\varSigma) \log|t| -  \frac{1}{2} \log \det \Im \varOmega(t)
\end{equation}
holds as $t \to 0$ over $\dd^*$. 
\end{thm}
As we will see in Remark~\ref{main_term}, the term $-I(\varSigma) \log|t| $ is the main term in (\ref{eqn:I-asymp}) if the family does not have good reduction over $\dd$.

\subsection{Preliminary observations}
Before giving the proof of Theorem~\ref{I-asymptotic} we discuss some useful explicit formulas for computing the $I$-invariant of a principally polarized complex abelian variety. 
Let $g \in \zz_{\geq 0}$ and let $\varOmega \in \mathbb{H}_g$.  As in \S\ref{sec:standard_form} let $A_\varOmega$ denote the principally polarized complex abelian variety $\cc^g/(\zz^g + \varOmega \zz^g)$ and let $\theta(z,\varOmega)$ be the associated Riemann theta function on $\cc^g$.
\begin{lem} \label{rewrite_I_inv} Let $F \subset \cc^g$ be any fundamental domain for the lattice $\zz^g + \varOmega \zz^g$. Let $\mu_F$ denote the Lebesgue measure on $\cc^g$ giving $F$ volume one.
 The formula
\[  \begin{split} 2\, & I(A_\varOmega)   + \frac{g}{2} \log 2  + \frac{1}{2} \log \det \Im \varOmega \\
&  =  2\pi  \int_{ F }  (\Im z)^t (\Im \varOmega)^{-1} (\Im z) \, \d \, \mu_F(z)  - 2 \, \int_{F} \log |\theta(z,\varOmega)|  \, \d \, \mu_F(z) \end{split} \]
holds.
\end{lem}
\begin{proof} 
 Let $\mu_H$ denote the Haar measure on $A_\varOmega$ giving $A_\varOmega$ volume one. Then by (\ref{norm_theta}) and (\ref{def_arch_lambda_bis}) we compute
\[ \begin{split} 
2\, I(A_\varOmega) & + \frac{g}{2} \log 2  + \frac{1}{2} \log \det \Im \varOmega \\
& = -2 \int_{A_\varOmega} \log \| \theta \| \, \d \, \mu_H + \frac{1}{2} \log \det \Im \varOmega \\
& = 2 \,  \int_{A_\varOmega} \left( \pi (\Im z)^t (\Im \varOmega)^{-1}  (\Im z) - \log |\theta(z,\varOmega) | \right) \, \d \, \mu_H(z) \\
& = 2\pi  \int_{ F }  (\Im z)^t (\Im \varOmega)^{-1} (\Im z) \, \d \, \mu_F(z) - 2 \int_{F } \log |\theta(z,\varOmega)|  \, \d \, \mu_F(z) 
 \, . \end{split} \]
The lemma follows.
\end{proof}
\begin{definition}
Let $W$ be any fundamental domain for the lattice $\zz^g$ in $\rr^g$.
We define the set
\begin{equation} \label{def_FE} F(W,\varOmega) = \left\{ \alpha + \varOmega \beta \, : \, \alpha \in \left[-\frac{1}{2},\frac{1}{2} \right]^g \, , \, \beta \in W \right\} \subset \cc^g \, . 
\end{equation}
It is easy to see that $F(W,\varOmega)$  is a fundamental domain for the lattice $\zz^g + \varOmega \zz^g$ in $\cc^g$.
\end{definition}
\begin{lem}  \label{eqn:relate_to_IV} 
Let $\mu_{W,\varOmega}$ denote the Lebesgue measure on $\cc^g$ giving $F(W,\varOmega)$ volume one.  The formula
\[ \int_{ F(W,\varOmega)}  (\Im z)^t (\Im \varOmega)^{-1} (\Im z) \, \d \, \mu_{W,\varOmega}(z) = I_W(\Im \varOmega) \]
holds. Here $I_W(\Im \varOmega)$ is defined as in (\ref{def:IV}). 
\end{lem}
\begin{proof} The measure $\mu_{W,\varOmega}$ is the pushforward of the standard Lebesgue measure on $\rr^g \times \rr^g$ onto $\cc^g$ along the $\rr$-linear isomorphism $(\alpha,\beta) \mapsto \alpha+\varOmega \beta$. This allows us to compute
\[ \begin{split}  \int_{ F(W,\varOmega)}  (\Im z)^t (\Im \varOmega)^{-1} (\Im z) \, \d \, \mu_{W,\varOmega}(z) & =
\int_{\alpha \in [-\frac{1}{2},\frac{1}{2}]^g } \int_{\beta \in W} \beta^t (\Im \varOmega) \beta \, \d \, \alpha \, \d \, \beta\\
& =  I_W(\Im \varOmega) \, . 
\end{split}  \]
The lemma follows.
\end{proof}

\subsection{Proof of Theorem~\ref{I-asymptotic}}
We now return to the setting of Theorem~\ref{I-asymptotic}. Thus,
we let $(f \colon \mathcal{A} \to \dd^*, \alpha)$ be a family of principally polarized complex abelian varieties over the punctured unit disk, with unipotent monodromy around the origin.  Let $r $ be the toric rank of the fiber at the origin of the family of semiabelian varieties determined by~$f$,   and let $\varSigma$ denote the polarized real torus determined by the principally polarized tropical abelian variety associated to~$f$.

As follows from  \S\ref{sec:NilpOrb} we may assume upon choosing a symplectic basis of a fiber of $R^1f_* \zz_\mathcal{A}$ that the $\mathcal{A}_t$ are given as the principally polarized abelian varieties $A_{\varOmega(t)}$ with
\begin{equation} \label{family_bis} \varOmega(t) = \frac{1}{2\pi i} A \log t + B(t)  \, , \quad t \in \dd^* \, , 
\end{equation}
where $B \colon \dd \to \Mat(g \times g,\cc)$ is a bounded holomorphic map and where $A$ is a symmetric integral $g\times g$ matrix in block form
\[ A = \begin{pmatrix}[c|c] A_0 & 0 \\
\hline 
0 & 0 \\
\end{pmatrix}  \, , \]
with $A_0$ positive definite and of size $r\times r$.  

For every $t \in \dd$ we denote by $B_{g-r,g-r}(t)$   the lower right $(g-r)\times(g-r)$ block of the matrix $B(t)$. By item (b) from \S\ref{sec:NilpOrb} we have that $B_{g-r,g-r}(0) \in \mathbb{H}_{g-r}$. 

\begin{definition}
We define
 \[ V =  \Vor(A_0) \times \Vor(\Im B_{g-r,g-r}(0)) \subset \rr^r \times \rr^{g-r} = \rr^g \, . \] 
It follows from Lemma~\ref{IVA_IVQ0}  that $V$ is a fundamental domain for the lattice~$\zz^g$ in~$\rr^g$.
\end{definition}

The key to our proof of Theorem~\ref{I-asymptotic} is the following proposition.
\begin{prop} \label{asympt_log_theta} Let $\varOmega(t)$ be as in (\ref{family_bis}). 
The fiber integral
\begin{equation} \label{is_continuous}
 \int_{F(V,\varOmega(t))} \log \left| \theta(z,\varOmega(t)) \right| \, \d \, \mu_{V,\varOmega(t)}(z)  
\end{equation}
extends continuously over $\dd$. Here $F(V,\varOmega(t))$ is defined as in (\ref{def_FE}), and $\mu_{V,\varOmega(t)}$ is the Lebesgue measure on $\cc^g$ giving $F(V,\varOmega(t))$ volume one.
\end{prop}
Before giving the proof of Proposition~\ref{asympt_log_theta} we first examine the result in the case that $g=1$, and show how Theorem~\ref{I-asymptotic} follows from the proposition.
\begin{example}
Assume $g=1$. For $\varOmega \in \mathbb{H}$ and $z \in \cc$ we set $t=\exp(2\pi i \varOmega)$ and $w=\exp(2\pi i z)$. The Riemann theta function admits a product expansion
\begin{equation} \label{product_exp}
 \theta(z,\varOmega) = \prod_{k=1}^\infty (1-t^k) \prod_{k=0}^\infty (1+t^{k+1/2}w^{-1}) (1+ t^{k+1/2} w) \, . 
\end{equation}
We have $V = \left[-\frac{1}{2},\frac{1}{2} \right] \subset \rr$. A computation based on (\ref{product_exp}) gives
\[ \int_{F(V,\varOmega)} \log |\theta(z,\varOmega) | \, \d \, \mu_{V,\varOmega}(z) = \log \left|  \prod_{k=1}^\infty (1-t^k) \right| \, . \]
The right hand side clearly extends continuously over $t=0$. 
\end{example}
Assuming Proposition~\ref{asympt_log_theta}, Theorem~\ref{I-asymptotic} can be proved as follows. 
\begin{proof}[Proof of Theorem~\ref{I-asymptotic}]  Combining Lemma~\ref{rewrite_I_inv} and Lemma~\ref{eqn:relate_to_IV} we have\begin{equation} \label{proof_I} \begin{split} 2\, I(A_{\varOmega(t)} )  & + \frac{g}{2} \log 2  + \frac{1}{2} \log \det \Im \varOmega(t) \\
& =  2\pi \,  \int_{ F(V,\varOmega(t))}  (\Im z)^t (\Im \varOmega(t))^{-1} (\Im z) \, \d \, \mu_{V,\varOmega(t)}(z) \\
 & \hspace{1cm} - 2 \int_{F(V,\varOmega(t))} \log |\theta(z,\varOmega(t))|  \, \d \, \mu_{V,\varOmega(t)}(z) \\ 
& = 2\pi\, I_V(\Im \varOmega(t)) - 2 \int_{F(V,\varOmega(t))} \log |\theta(z,\varOmega(t))|  \, \d \, \mu_{V,\varOmega(t)}(z) \, . \end{split} 
\end{equation}
As the map $B \colon \dd \to \Mat(g \times g,\cc)$ is bounded and holomorphic, we have
\begin{equation}  \begin{split}  2\pi\, I_V(\Im \varOmega(t) ) & = - I_V(A) \log|t| + 2\pi  \,  I_V( \Im B(t))  \\
 & \sim - I_V(A) \log|t| \, . \end{split} 
\end{equation}
By Lemma~\ref{IVA_IVQ0} and Proposition~\ref{monodromy-Gram-matrix} 
we have
\begin{equation}
 I_V(A)  = I(A_0) = I(\varSigma) \, . 
\end{equation}
We conclude that
\begin{equation} \label{proof_III} 2\pi\, I_V(\Im \varOmega(t) ) \sim -I(\varSigma) \log|t| \, . 
\end{equation}
By Proposition~\ref{asympt_log_theta} we have that 
\begin{equation} \label{proof_II} \int_{F(V,\varOmega(t))} \log |\theta(z,\varOmega(t))|  \, \d \, \mu_{V,\varOmega(t)}(z)  \sim 0  \, .
\end{equation}
Combining (\ref{proof_I}),  (\ref{proof_III}) and (\ref{proof_II}) gives that
\[ 2\,  I(A_{\varOmega(t)}) \sim  -I(\varSigma) \log|t| -  \frac{1}{2} \log \det \Im \varOmega(t) \, , \]
and the theorem is proven.
\end{proof}
\begin{remark} \label{main_term} From (\ref{family_bis}) we obtain that $\det \Im \varOmega(t) \sim c \cdot (-\log|t|)^r$ for a suitable real number $c>0$. This gives 
\[ 2\,  I(A_{\varOmega(t)}) \sim  -I(\varSigma) \log|t| -  \frac{r}{2} \log (- \log |t|) \, , \]
showing that $-I(\varSigma) \log|t| $ is the main term of the asymptotics if $r>0$.
\end{remark}

\subsection{Proof of Proposition~\ref{asympt_log_theta}}
We will deduce Proposition~\ref{asympt_log_theta}  from the following  lemmas.

\begin{definition} \label{def:analytic_harm} For $U \subset \rr^n$ open we call $f \colon U \to \cc$ \emph{analytic} (resp.\ \emph{harmonic}) if both the real and imaginary part of $f$ are real analytic (resp.\ harmonic).  
\end{definition}

\begin{lem} \label{generalities} Let $U \subset \rr^n$ be an open set. (i) Each harmonic function on $U$ is analytic. 
(ii) Let $D \subset U$ be a closed polar set. Then each harmonic function on $U \setminus D$ which is locally bounded on $U $ extends to a harmonic function on $U$.
\end{lem}
\begin{proof} Both statements follow directly from the corresponding properties of real-valued harmonic functions. For (i) we refer to  
 \cite[Theorem~1.28]{hft}, and for (ii) we refer to \cite[Chapter 3, \S2]{brelot}.
\end{proof}

The next lemma can be proved using the Weierstrass Preparation Theorem for real analytic functions \cite[Theorem~6.1.3]{primer} and continuity of the roots of a polynomial. 
We refer to \cite[Lemma~1.5.3]{bgs} and \cite[Lemma~6.6]{we} for similar statements in a holomorphic setting. 
\begin{lem} \label{continuity_fiber_integral} Let $U \subset \rr^n$ be an open set and let $M \subset \rr^m $ be a bounded open set.
Let $f \colon M \times U \to \cc$ be a bounded analytic  function in the sense of Definition~\ref{def:analytic_harm}. 
 Let $ \mu_L$ denote the standard Lebesgue measure on $\rr^m$. Assume that, for all $t \in U$, the function $f(x,t)$ is not identically zero in $x \in M$. The fiber integral
\[ \int_{x \in M} \log | f(x,t) | \, \d \, \mu_L(x) \]
is a continuous function of $t \in U$. 
\end{lem}

\begin{proof}[Proof of Proposition~\ref{asympt_log_theta}]   
For $(\alpha, \beta) \in \rr^g \times \rr^g$ and $t \in \dd^*$ we define
\[ \phi(\alpha,\beta,t) = \theta(\alpha+ \varOmega(t) \beta,\varOmega(t)) \, . \]
Since $\theta$ and $\varOmega$ are holomorphic functions we find that $\phi$ is harmonic in the sense of Definition~\ref{def:analytic_harm}. 
The measure $\mu_{V,\varOmega(t)}$ is the pushforward of the standard Lebesgue measure on $\rr^g \times \rr^g$ onto $\cc^g$ along the $\rr$-linear isomorphism $(\alpha,\beta) \mapsto \alpha+\varOmega(t) \beta$, and for each $t \in \dd^*$ the integral in  (\ref{is_continuous}) can be performed over the interior of $F(V,\varOmega(t))$ as well. Thus, letting $V^{\mathrm{int}}$ denote the interior of $V$, and letting $f_0$ denote the restriction of $\phi$ to $ \left(-\frac{1}{2},\frac{1}{2} \right)^g \times V^{\mathrm{int}} \times \dd^*$ it follows that for $t \in \dd^*$ the integral in (\ref{is_continuous}) is equal to the integral
 \[ \int_{(\alpha, \beta) \in \left(-\frac{1}{2},\frac{1}{2} \right)^g \times V^{\mathrm{int}}} \log |f_0(\alpha,\beta,t)| \, \d \, \alpha \, \d \, \beta \, . \]
In order to prove Proposition~\ref{asympt_log_theta} it suffices, by Lemma~\ref{continuity_fiber_integral}, to show that: 
\begin{itemize}
\item[(i)] $f_0$ extends to a bounded analytic function $f$ over $  \left(-\frac{1}{2},\frac{1}{2} \right)^g \times V^{\mathrm{int}} \times \dd$; 
\item[(ii)] the restriction of $f$ to the fiber at the origin is not identically zero. 
\end{itemize}
For a vector $n \in \rr^g$ we shall denote by $n^{(r)} \in \rr^r$ the first~$r$ entries, and by $n^{(g-r)} \in \rr^{g-r}$ the last $g-r$ entries.

\medskip

\noindent \emph{Proof of item (i).} 
Using (\ref{family_bis}) and the definition of the Riemann theta function in~(\ref{def:Riemann_theta}) we find for $f_0(\alpha,\beta,t)$ the formal series expansion
\begin{equation} \label{series_f}
  \sum_{n^{(r)} \in \zz^r} t^{\frac{1}{2} n^{(r),t} A_0 n^{(r)} + n^{(r),t} A_0 \beta^{(r)} } \sum_{n^{(g-r)} \in \zz^{g-r}} \exp( \pi i  n^t B(t) n + 2 \pi i  n^t  (\alpha + B(t) \beta ) )   .
  \end{equation}
By (\ref{when_in_Vor}) we see that for all $\beta \in V^{\mathrm{int}}$ and $n \in \zz^g$ the exponent $\frac{1}{2} n^{(r),t} A_0   n^{(r)} + n^{(r),t}  A_0 \beta^{(r)} $ of the variable~$t$ is non-negative. 
By item (a) from \S\ref{sec:NilpOrb} we may assume that  there exists $c>0$ such that $\Im B_{g-r,g-r}(t) > c \cdot \Id_{g-r,g-r}$ for $t \in \dd$. Thus there exists $C \in \rr$ such that on $\left(-\frac{1}{2},\frac{1}{2} \right)^g \times V^{\mathrm{int}} \times \dd$ the estimate
\[ \left| \sum_{n^{(g-r)} \in \zz^{g-r}} \exp( \pi i \, n^t  B(t)   n + 2 \pi i \, n^t  (\alpha + B(t)   \beta ) )  \right| \leq \exp( C  \| n^{(r)} \|^2) \]
holds. By this estimate we may conclude that the series (\ref{series_f}) converges absolutely at each point of $\left(-\frac{1}{2},\frac{1}{2} \right)^g \times V^{\mathrm{int}} \times \dd$. We see that the function $f_0$ is  bounded on $\left(-\frac{1}{2},\frac{1}{2} \right)^g \times V^{\mathrm{int}} \times \dd^*$ and extends to a bounded function $f$ on $\left(-\frac{1}{2},\frac{1}{2} \right)^g \times V^{\mathrm{int}} \times \dd$. As $f_0$ is harmonic, by Lemma~\ref{generalities}(ii) we find that the function $f$ is harmonic. As harmonic functions are analytic by Lemma~\ref{generalities}(i) we are done.

\medskip

\noindent \emph{Proof of item (ii).} The restriction of $f$ to the fiber at the origin is given by (\ref{series_f}), specialized to $t=0$. For $\beta \in V^{\mathrm{int}}$ and $n^{(r)} \in \zz^r $ we have $\frac{1}{2} n^{(r),t} A_0  n^{(r)} + n^{(r),t}  A_0   \beta^{(r)} > 0$ unless $n^{(r)}=0$.  Taking only the contributions with $n^{(r)} = 0$ in the series (\ref{series_f}) we see that the restriction of $f$ to the fiber at the origin equals 
\[ \begin{split} \sum_{n^{(g-r)} \in \zz^{g-r}} \exp( \pi i \, n^{(g-r),t}   & B_{g-r,g-r}(0)  n^{(g-r)} + 2 \pi i \, n^{(g-r),t}   (\alpha + B(0) \beta )^{(g-r)} ) \\ 
& =  \theta( (\alpha + B(0)\beta)^{(g-r)}, B_{g-r,g-r}(0) ) \, . 
\end{split} \]
This is not identically zero.
\end{proof}

\subsection{The constant term in the asymptotics of the $I$-invariant} \label{sec:cst_term}
We finish this section by making the constant term implied by Theorem~\ref{I-asymptotic} explicit. We think that our explicit expression for the constant term may have further applications. The end result is written in (\ref{cst_term}). The right hand side of (\ref{cst_term}) is manifestly a function of the limit mixed Hodge structure associated to $f$, together with its monodromy action. We refer to \cite[Theorem~76]{bp} for a similar result in the setting of limits of heights of biextensions.

We will write $Q(t) = 2\pi \, \Im B(t)$. Let $P$ be the principally polarized complex abelian variety of dimension~$g-r$ arising from the $\Gr_{-1}$ of the limit mixed Hodge structure associated to $f$. By item (b) from \S\ref{sec:NilpOrb} the matrix $B_{g-r,g-r}(0)$ is a period matrix of $P$.
 Looking at the various steps in the proof of Theorem~\ref{I-asymptotic} we find
\[ \begin{split} \lim_{t \to 0} & \left[   2\,  I(\mathcal{A}_t,\alpha_t) +  I(\varSigma) \log|t| +  \frac{1}{2} \log \det \Im \varOmega(t)
  \right] \\ 
  & = -  \frac{g}{2}\log 2 + I_V(Q(0)) \\ 
  & \hspace{1cm} - 2 \int_{\left[-\frac{1}{2},\frac{1}{2} \right]^g} \int_V \log| \theta( (\alpha + B(0)  \beta)^{(g-r)},B_{g-r,g-r}(0))| \, \d \, \alpha \, \d \, \beta \, . 
  \end{split} \]
Using Lemma~\ref{IVA_IVQ0} we can rewrite the right hand side as
\[ \begin{split}
   -  \frac{g}{2}& \log 2  + I(Q_{g-r,g-r}(0)) + I_{\Vor(A_0)}\left( Q_{r,r}(0)  \right) \\
  & \hspace{1cm} - 2 \int_{\left[-\frac{1}{2},\frac{1}{2} \right]^g} \int_V \log| \theta( (\alpha + B(0)  \beta)^{(g-r)},B_{g-r,g-r}(0))| \, \d \, \alpha \, \d \, \beta \, . \end{split} \]
 We note that for $\alpha, \beta \in \rr^g$
\[ (\alpha + B(0)  \beta)^{(g-r)} = \alpha^{(g-r)} + B_{g-r,r}(0)  \beta^{(r)} + B_{g-r,g-r}(0)  \beta^{(g-r)} \, . \] 
The double integral can therefore be written as a triple integral 
\[ \begin{split}  \int_{\left[-\frac{1}{2},\frac{1}{2} \right]^{g-r}} & \int_{\Vor(A_0)}  \int_{ \Vor(Q_{g-r,g-r}(0))} \cdots \\
& \hspace{1cm} \cdots   \log| \theta( (\alpha + B(0)  \beta)^{(g-r)},B_{g-r,g-r}(0))| \, \d \, \alpha^{(g-r)} \, \d \, \beta^{(r)} \, \d \, \beta^{(g-r)} \, . 
\end{split} \] 
In order to evaluate the triple integral we fix an element $\beta^{(r)} \in \Vor(A_0)$ and perform the integration over $\alpha^{(g-r)}$, $\beta^{(g-r)}$ first. We set
\[ \begin{split} F(\beta^{(r)})  & = \left\{ \alpha^{(g-r)}  + B_{g-r,r}(0)  \beta^{(r)} + B_{g-r,g-r}(0)  \beta^{(g-r)} \, \Big| \right. \\
& \hspace{2cm} \left. \alpha^{(g-r)} \in \left[-\frac{1}{2},\frac{1}{2} \right]^{g-r} \, , \, \beta^{(g-r)} \in \Vor(Q_{g-r,g-r}(0)) \right\} \subset \cc^{g-r} \, . \end{split} \]
The set $F(\beta^{(r)})$ is a fundamental domain for the lattice $\zz^{g-r} + B_{g-r,g-r}(0) \zz^{g-r}$ in $\cc^{g-r}$. Let $\mu_{\beta^{(r)}}$ denote the Lebesgue measure on  $\cc^{g-r}$ giving $F(\beta^{(r)})$ volume one. Still keeping $\beta^{(r)} \in \Vor(A_0)$ fixed we find using Lemma~\ref{rewrite_I_inv}
\[ \begin{split}-2  \int_{ \left[-\frac{1}{2},\frac{1}{2} \right]^{g-r} } & \int_{ \Vor(Q_{g-r,g-r}(0))} \cdots \\
 & \hspace{1cm} \cdots \log| \theta( (\alpha + B(0)  \beta)^{(g-r)},B_{g-r,g-r}(0))| \, \d \, \alpha^{(g-r)} \, \d \, \beta^{(g-r)} \\ & = -2 \int_{F(\beta^{(r)})}  \log |\theta(w, B_{g-r,g-r}(0))| \, \d \, \mu_{\beta^{(r)}} (w) \\
& = 2\, I( P ) + \frac{g-r}{2} \log 2 + \frac{1}{2} \log \det \Im B_{g-r,g-r}(0) \\ 
& \hspace{1cm} -2\pi \int_{F(\beta^{(r)})} 
(\Im w)^t  (\Im B_{g-r,g-r}(0) )^{-1}  (\Im w) \, \d \, \mu_{\beta^{(r)}}(w) \, . 
\end{split} \]
Now we note that
\[ \begin{split} 2\pi \, \Im & ( \alpha^{(g-r)}  + B_{g-r,r}(0)  \beta^{(r)} + B_{g-r,g-r}(0)  \beta^{(g-r)} ) \\
 & = Q_{g-r,r}(0)  \beta^{(r)} + Q_{g-r,g-r}(0)  \beta^{(g-r)} \, . \end{split} \]
This gives
\[ \begin{split} 2\pi & \int_{F(\beta^{(r)})}  (\Im w)^t  (\Im B_{g-r,g-r}(0) )^{-1}  (\Im w) \, \d \, \mu_{\beta^{(r)}}(w) \\
& = I(Q_{g-r,g-r}(0)) + \beta^{(r),t}   Q_{r,g-r}(0)  Q_{g-r,g-r}^{-1}(0)  Q_{g-r,r}(0)  \beta^{(r)} \, . \end{split} \]
For the triple integral we therefore find 
\[ \begin{split}  -2 & \int_{\left[-\frac{1}{2},\frac{1}{2} \right]^{g-r}}  \int_{\Vor(A_0)} \int_{ \Vor(Q_{g-r,g-r}(0))} \cdots \\
 & \hspace{1cm} \cdots \log| \theta( (\alpha + B(0)  \beta)^{(g-r)},B_{g-r,g-r}(0))| \, \d \, \alpha^{(g-r)} \, \d \, \beta^{(r)} \, \d \, \beta^{(g-r)} \\
 & = 2 \, I(P ) + \frac{g-r}{2} \log 2 + \frac{1}{2} \log \det \Im B_{g-r,g-r}(0)  \\ 
 & \hspace{1cm} - I(Q_{g-r,g-r}(0)) - I_{\Vor(A_0)}(Q_{r,g-r}(0)  Q_{g-r,g-r}^{-1}(0)  Q_{g-r,r}(0)) \, . \end{split} \]
This leads to
  \begin{equation} \label{cst_term} \begin{split} \lim_{t \to 0} & \left[   2\,  I(\mathcal{A}_t,\alpha_t) +  I(\varSigma) \log|t| +  \frac{1}{2} \log \det \Im \varOmega(t)
  \right] \\ 
  & = -\frac{r}{2} \log 2 + 2 \, I(P) + \frac{1}{2} \log \det \Im B_{g-r,g-r}(0) \\
  & \hspace{1cm} + I_{\Vor(A_0)} \left( Q_{r,r}(0) - Q_{r,g-r}(0)  Q_{g-r,g-r}^{-1}(0)  Q_{g-r,r}(0) \right) \, . \end{split}
\end{equation}
We note that $Q_{r,r}(0) - Q_{r,g-r}(0)  Q_{g-r,g-r}^{-1}(0)  Q_{g-r,r}(0)$ is a \emph{Schur complement} of the matrix $Q(0)$.
It would be interesting to further interpret the term in (\ref{cst_term}) related to this Schur complement. 

\begin{example}
Assume that $r=1$ and that the family $\mathcal{A}_t$ is the family of Jacobians associated to a family of compact Riemann surfaces of genus~$g>0$, whose limit at $t=0$ is the stable curve obtained by taking a compact Riemann surface $M$ of genus $g-1$ and identifying two distinct points $a, b \in M$. Then $P=\Jac M$, and the vector $B_{1,g-1}$ represents the point in $\Jac M$ determined by the divisor $a-b$. Let $\ee(x,y)$ denote the Riemann prime form of $M$, and consider the real-valued variant
\[ F(x,y) = \exp(-2\pi \,\Im(x-y)^t  (\Im B_{g-1,g-1}(0))^{-1}  \Im(x-y) ) |\ee(x,y)|^2 \]
as in \cite[\S2]{we}. Following the proof of \cite[Lemma~7.5]{we} we have $ Q_{1,1}(0) = 2 \log |\ee(a,b)| $ and hence we find for the Schur complement
\[ \begin{split}  Q_{1,1}(0) - & Q_{1,g-1}(0)  Q_{g-1,g-1}^{-1}(0)  Q_{g-1,1}(0) \\
& = 2 \log |\ee(a,b)| - 2\pi \, \Im(a-b)^t   (\Im B_{g-1,g-1}(0))^{-1}  \Im(a-b) \\
& = \log |F(a,b)| \, . 
\end{split} \]
We have $\Vor(A_0) =  \left[-\frac{1}{2},\frac{1}{2} \right] \subset \rr$ and hence
\[ I_{\Vor(A_0)}  \left( Q_{1,1}(0) -  Q_{1,g-1}(0)  Q_{g-1,g-1}^{-1}(0)  Q_{g-1,1}(0) \right) = \frac{1}{12} \log |F(a,b)| \, . \]
\end{example}

\section{The archimedean $\lambda$-invariant and its asymptotics} \label{curves}

In this section we prove Theorem~\ref{lambda_asympt_intro}. We will be brief here, as a very similar reasoning is also applied in \cite[\S7]{wi_deg}. In fact, Theorem~\ref{phi_asympt} below coincides with \cite[Theorem~1.2]{wi_deg}. 

\subsection{The archimedean $\lambda$-invariant} \label{sec:arch_lambda}
Let $g \in \zz_{>0}$.
Let $C$ be a compact and connected Riemann surface of genus $g$.  Let $\delta_F(C)$ be the Faltings delta-invariant of $C$ as defined in \cite[p.~401]{fa} and let $\varphi(C)$ be the Zhang-Kawazumi invariant of $C$ as defined in \cite[\S1.3]{zhgs}. Up to a multiplicative constant, the invariant $\varphi(C)$ is the same as the invariant $a(C)$ introduced and studied independently by N.\ Kawazumi in \cites{kaw, kawhandbook}. 
Put $\delta(C)=\delta_F(C)-4g\log(2\pi)$.

\medskip

\noindent \textbf{Definition.}
The $\lambda$-invariant $\lambda(C)$ of $C$ is defined to be the real number 
 \begin{equation} \label{lambda_arch}  \lambda(C) = \frac{g-1}{6(2g+1)} \varphi(C) + \frac{1}{12}\delta(C) \, ,
 \end{equation}
see also \cite[\S1.4]{zhgs}.

\medskip

Let $\Jac C$ denote the Jacobian of $C$, seen as a principally polarized complex abelian variety.
From \cite[Theorem~1.1]{wi} and  (\ref{def_arch_lambda_bis}) we arrive at the identity
\begin{equation} \label{eq:wilms}
   2 \, \varphi(C) = \delta_F(C) - 24 \, I(\Jac(C)) + 8g \log(2\pi) - 6g \log 2 \, . 
\end{equation}

\subsection{Invariants of polarized metric graphs}
Let $\varGamma$ be a connected metric graph as in \S\ref{pm_graphs}. It is shown in \cite[Theorem 2.11]{cr} that there exists a unique measure $\mu_{\can}$ on $\varGamma$ having total volume one, such that $g_{\mu_{\can}}(x, x)$ is a constant. Here $g_{\mu} \colon \varGamma \times \varGamma \to \rr$ is the Green's function determined by a measure $\mu$. We define $\tau(\varGamma)$ to be this constant. We further let $\delta(\varGamma)$ denote the total length of $\varGamma$, and we write $\Jac(\varGamma)$ for the tropical Jacobian of $\varGamma$.
In \cite[Theorem~B]{djs} we have shown the equality
\begin{equation} \label{eqn:universal_linear}
 I(\Jac(\varGamma)) + \frac{1}{2}\tau(\varGamma) = \frac{1}{8}\delta(\varGamma) \, . 
\end{equation}
 Let $\overline{\varGamma}=(\varGamma,\mathbf{q})$ be a polarized metric graph. Let $\epsilon(\overline{\varGamma})$ be its $\epsilon$-invariant (\ref{eps_invariant}). By \cite[Proposition~9.2]{dj_NT_height} we have
\begin{equation} \label{eqn:epsilon_phi}
 \frac{1}{12} \left( \delta(\varGamma) + \epsilon(\overline{\varGamma}) - 2 \, \varphi(\overline{\varGamma}) \right) + \frac{1}{2} \tau(\varGamma) = \frac{1}{8} \delta(\varGamma) \, . 
\end{equation}
By combining (\ref{eqn:universal_linear}) and (\ref{eqn:epsilon_phi}) we arrive at  the equality
\begin{equation} \label{eq:djs}
 2\,\varphi(\overline{\varGamma}) = \delta(\varGamma) +\epsilon(\overline{\varGamma}) -12 \, I(\Jac(\varGamma)) \, . 
 \end{equation}
 
 \subsection{Proof of Theorem~\ref{lambda_asympt_intro}}
Let $\pi \colon \mathcal{X} \to \dd$ be a stable curve of genus $g \geq 2$ over the open unit disc $\dd$. We assume that $\pi$ is smooth over $\dd^*$.  Let $\overline{G} = (G,\mathbf{q})$ be the  polarized weighted graph associated to $\pi$ obtained by taking as underlying polarized graph the dual graph of the special fiber $\mathcal{X}_0$ and endowing the vertices with the arithmetic genera of the corresponding irreducible components, and by letting the lengths of the edges be given by the ``thicknesses'' of the corresponding singular points of $\mathcal{X}_0$ on the total space $\mathcal{X}$. Let $\overline{\varGamma}= (\varGamma,\mathbf{q})$ be the  polarized metric graph associated to $\overline{G}=(G,\mathbf{q})$.

Write $X=\pi^{-1}\dd^*$.
Let $f \colon \mathcal{J} \to \dd^*$  denote the  family of Jacobians associated to the family of curves $X \to \dd^*$. Then $\mathcal{J}$ extends to a family of semiabelian varieties over $\dd$. 
Let $\varOmega(t)$ be the family of period matrices determined by a symplectic basis of a fiber $H$ of  $R^1\pi_* \zz_X=R^1 f_*\zz_{\mathcal{J}}$. By \cite[Theorem~1.1]{dj_stable} we have the asymptotics
\begin{equation} \label{asymptdelta} \delta_F(X_t) \sim -(\delta(\varGamma)+\epsilon(\overline{\varGamma})) \log|t| - 6 \log \det \Im \varOmega(t) 
\end{equation}
for the Faltings delta-invariants of the curves $X_t$ as $t \to 0$. 

Following for example \cite[\S3.2]{abbf}, the character group $Y=\Gr_0 H$ from the principally polarized tropical abelian variety $(X,Y,\varPhi,b)$ associated to the degenerating family $\mathcal{J}$ of principally polarized abelian varieties is canonically identified with the first homology group $H_1(\varGamma,\zz)$. Moreover,  the associated inner product $[\cdot,\cdot]_Y$ on $Y_\rr$ is identified with the natural cycle pairing on $H_1(\varGamma,\rr)$, see \cite[Proposition~3.4]{abbf}.
This gives that the principally polarized real torus $\varSigma$ associated to $(X,Y,\varPhi,b)$  is canonically identified with the tropical Jacobian $\Jac(\varGamma)$ of $\varGamma$.

Using this, Theorem~\ref{I-asymptotic} specializes to saying that
\begin{equation} \label{I-Jac-asymp}
 2 \, I(\Jac X_t) \sim - I(\Jac(\varGamma)) \log|t| - \frac{1}{2} \log \det \Im \varOmega(t) 
\end{equation}
as $t \to 0$. 
Combining (\ref{eq:wilms}), (\ref{eq:djs}),  (\ref{asymptdelta}) and (\ref{I-Jac-asymp}) we obtain the following asymptotics  for the Zhang-Kawazumi invariant $\varphi$ in the family $X_t$. 
\begin{thm} \label{phi_asympt} 
Let $\varphi(\overline{\varGamma}) $ be Zhang's $\varphi$-invariant (\ref{phi_nonarch}) of the polarized metric graph $\overline{\varGamma}$. Then one has the asymptotics
\begin{equation}
\varphi(X_t) \sim - \varphi(\overline{\varGamma}) \log|t| 
\end{equation}
for the Zhang-Kawazumi invariants of the curves $X_t$ as $t \to 0$ over $\dd^*$. 
\end{thm}
The asymptotics for the archimedean $\lambda$-invariant  as displayed in Theorem~\ref{lambda_asympt_intro} follows upon combining Theorem~\ref{phi_asympt} with (\ref{lambda_nonarch}), (\ref{lambda_arch})  and  (\ref{asymptdelta}).

\section{Hodge metric and its asymptotics} \label{sec:hodge}

In our proofs of Theorem~\ref{main:intro} and Theorem~\ref{height_Ceresa_intro} we need a couple of well-known facts about the Hodge metric on the determinant of the Hodge bundle. We collect them here. 

\subsection{Determinant of the Hodge bundle on $\aa_g$} \label{sec:det_Hdg_Ag}
Let $\aa_g = \Sp(2g,\zz) \setminus \mathbb{H}_g$ be the moduli space of principally polarized complex abelian varieties, viewed as an orbifold, with projection map $q \colon \mathbb{H}_g \to \aa_g$. 
Let $f \colon \mathcal{U}_g \to \aa_g$ denote the universal abelian variety, and let $\ll_f =  \det f_* \varOmega^1_{\mathcal{U}_g/\aa_g}$ be the determinant of the Hodge bundle on $\aa_g$. 
 The line bundle $q^*\ll_f$ is trivialized by the frame $\omega =  (2\pi i)^g  (\d z_1 \wedge \cdots \wedge \d z_g)$, where $(z_1,\ldots,z_g)$ are the standard Euclidean coordinates on $\cc^g$. The Hodge metric on $q^*\ll_f$  is given explicitly by the formula
\begin{equation} \label{norm_frame}  \|   \d z_1 \wedge \cdots \wedge \d z_g \|_{\Hdg}(\varOmega) = ( \det \Im \varOmega)^{1/2} \, , \quad \varOmega \in \mathbb{H}_g \, . 
\end{equation}
The metric $\|\cdot\|_{\Hdg}$  descends along $q$ to give a smooth hermitian metric  on $\ll_f$ that we also denote by $\|\cdot\|_{\Hdg}$.  When $h \in \zz $ and $\xi$ is a local meromorphic section of the line bundle $\ll_f^{\otimes h}$ over $\aa_g$ we set
\begin{equation} \label{def:tilde}
  \widetilde{\xi} = q^*(\xi) \, \omega^{\otimes -h} = (2\pi i)^{-gh} \, q^*(\xi) \, (\d z_1 \wedge \cdots \wedge \d z_g)^{\otimes - h} \, . 
\end{equation}
We can view $\widetilde{\xi}$ canonically as a local meromorphic function on $\mathbb{H}_g$. 

It follows from (\ref{norm_frame}) that when $(A,a)$ is a principally polarized complex abelian variety, and $\xi$ is a local meromorphic section of the line bundle $\ll_f^{\otimes h}$ near the moduli point of $(A,a)$ in $\aa_g$, and $\varOmega$ is a period matrix of $(A,a)$, we have the equality
\begin{equation} \label{def_norm}   \|\xi\|_{\Hdg}(A,a) = (2\pi)^{gh} \, |\widetilde{\xi}| (\varOmega) \, (\det \Im \varOmega)^{h/2}   
\end{equation}
in $\rr$.

\subsection{Determinant of the Hodge bundle on $\mm_g$} \label{sec:det_Hdg_Mg}
Let $p \colon \mathcal{C}_g \to \mm_g$ denote the universal Riemann surface over $\mm_g$. We denote by $\ll_p = \det p_* \omega_{\mathcal{C}_g / \mm_g}$ the determinant of the Hodge bundle on $\mm_g$. As was discussed in the introduction the line bundle $\ll_p$ comes equipped with a natural smooth hermitian metric  derived from the inner product given by $(\alpha, \beta) \mapsto \frac{i}{2} \int_C \alpha \wedge \overline{\beta}$ on the space of holomorphic $1$-forms on a compact connected Riemann surface $C$. 

Letting $\mathrm{t} \colon \mm_g \to \aa_g$ denote the Torelli map, we have a canonical isomorphism of holomorphic line bundles $\ll_p \isom \mathrm{t}^* \ll_f$ on $\mm_g$, where $\ll_f$ is the determinant of the Hodge bundle on $\mathcal{A}_g$ as in \S\ref{sec:det_Hdg_Ag}. This isomorphism is an isometry when $\mathrm{t}^*\ll_f$ is equipped with the pullback of the Hodge metric 
$\| \cdot \|_{\mathrm{Hdg}}$.

\subsection{Extension over the boundary} \label{sec:extension}
Let $\overline{S}$ be a smooth complex algebraic variety,  let $D$ be a normal crossings divisor on $\overline{S}$ and write $S = \overline{S} \setminus D$. Let $\pi \colon \mathcal{X} \to \overline{S}$ be a stable curve of genus $g \geq 2$, and assume that the map $\pi$ is smooth over $S$. Let $\ll_S$ be the pullback of the  line bundle $\ll_p$ from \S\ref{sec:det_Hdg_Mg}  along the moduli map $S \to \mm_g$, equipped with the pullback of the metric $\| \cdot \|_{\mathrm{Hdg}}$.
Write $\ll_{\overline{S}} = \det \pi_* \omega_{\mathcal{X}/\overline{S}}$ where $\omega_{\mathcal{X}/\overline{S}}$ is the relative dualizing sheaf of~$\pi$, and let $X = \pi^{-1}S$.
\begin{prop} \label{mumford} (i) The first Chern form $c_1(\ll_S)$ determines a current $\left[ c_1(\ll_S) \right]$ over $\overline{S}$. (ii) In the case that $\overline{S}$ is projective, the current $\left[ c_1(\ll_S) \right]$ represents the cohomology class $c_1(\ll_{\overline{S}} )$. (iii) Assume that $\overline{S}=\dd$, and $D=\{0\}$. Let $\xi$ be a local generating section of $\ll_{\dd}$ near~$0$. Let $\varOmega(t)$ be the family of period matrices on $\dd^*$ determined by a symplectic framing of $R^1\pi_*\zz_X$. The asymptotic
\begin{equation} \label{hodge_asymp}
\log \| \xi \|_{\mathrm{Hdg}} \sim  \ord_0(\xi,\ll_{\dd}) \log|t| + \frac{1}{2} \log \det \Im \varOmega(t) 
\end{equation}
holds. 
\end{prop}
\begin{proof} By \cite[Theorem~3.1]{mu} the Hodge metric on $\ll_S$ is a good metric in the sense of Mumford. This implies item (i) by \cite[Proposition~1.1]{mu}. Also, by \cite[Proposition~1.3]{mu} the line bundle $\ll_S$ has a canonical extension $\overline{\ll_S}$ uniquely determined by the property that the norms of local generating sections have at most logarithmic growth in any local coordinate system. By \cite[p.~225]{fc} the canonical extension $\overline{\ll_S}$ is equal to $\ll_{\overline{S}}$. Item (ii) then follows upon applying \cite[Theorem~1.4]{mu}. 

As to item (iii),  by the results in \S\ref{sec:NilpOrb} based on the Nilpotent Orbit Theorem there exist $c, r \in \zz_{\geq 0}$ such that $ \det \Im \varOmega(t) \sim c \left(- \log |t| \right)^r $
as $t \to 0$. Combining with (\ref{norm_frame}) we conclude that $\varOmega^*( \d z_1 \wedge \ldots \wedge \d z_g) $ extends as a frame of the canonical extension $\overline{\ll_{\dd^*}}$ over $\dd$. As we have just seen, this canonical extension is equal to $\ll_{\dd}$. By (\ref{def:tilde}) we therefore obtain the equality $\ord_0(\xi,\ll_{\dd})=\ord_0(\tilde{\xi})$. This  leads to the asymptotic $\log |\tilde{\xi}| \sim \ord_0(\xi,\ll_{\dd}) \log|t|$ as $t \to 0$. The asymptotic (\ref{hodge_asymp}) follows from combining this with (\ref{def_norm}).
\end{proof}

\subsection{Connection with the $\lambda$-invariant} \label{sec:connection_lambda}
As was discussed in the introduction, the Hain-Reed line bundle $\bb$ and the line bundle $\ll_p^{\otimes 8g+4}$ on $\mm_g$ are isomorphic as holomorphic line bundles, and the set of isomorphisms $\phi \colon \bb \isom \ll_p^{\otimes 8g+4}$ is a $\cc^\times$-torsor. For each such isomorphism $\phi$ we have an $\rr$-valued function $\beta_\phi$ as in (\ref{def:beta})
on $\mm_g$. 

We recall that by \cite[Theorem~1.4]{dj_second}, for a suitable choice of the isomorphism $\phi$, one has the equality of functions $\beta_\phi = (8g+4) \lambda$ on $\mm_g$, where $\lambda \colon \mm_g \to \rr$ is the archimedean $\lambda$-invariant given in (\ref{lambda_arch}). In the following we will fix this isomorphism $\phi$. We thus obtain from (\ref{def:beta}) the useful identity
\begin{equation} \label{lambda_new_def}
 (8g+4) \lambda = \log \left( \frac{ \| \cdot \|_\bb }{\phi^* \| \cdot \|_{\Hdg}} \right)
\end{equation}
of functions on $\mm_g$. 

\section{Height of the Ceresa cycle in the function field case} \label{sec:height_Ceresa}

In this section we derive Theorem~\ref{height_Ceresa_intro} from Theorem~\ref{lambda_asympt_intro}. We repeat the setting: let $\overline{S}$ be a smooth projective connected complex curve, let $D$ be an effective reduced divisor on $\overline{S}$, and let $S = \overline{S} \setminus D$. We consider a  stable curve $\pi \colon \mathcal{X} \to \overline{S}$ of genus $g \geq 2$, smooth over $S$. Let $\bb_S$ denote the pullback of the Hain-Reed line bundle $\bb$ along the moduli map $S \to \mm_g$. 

As in \S\ref{sec:extension} let $\ll_S$ be the pullback of the determinant of the Hodge bundle along the moduli map $S \to \mm_g$, equipped with the pullback of the metric $\| \cdot \|_{\mathrm{Hdg}}$.
We write $\ll_{\overline{S}} = \det f_* \omega_{\mathcal{X}/\overline{S}}$ and let $h(\mathcal{X}/\overline{S})= \deg \ll_{\overline{S}}$ be the modular height of $f$.
\begin{proof}[Proof of Theorem~\ref{height_Ceresa_intro}] 
From (\ref{lambda_new_def}) we obtain an equality of $(1,1)$-forms
\begin{equation} \label{c1_overS}
 (8g+4) c_1(\ll_S) = c_1(\bb_S) - (8g+4) \frac{\partial \overline{\partial}}{\pi i} \lambda 
\end{equation}
over $S$.
By Proposition~\ref{mumford} the first Chern form $c_1(\ll_S)$ determines a current $\left[c_1(\ll_S)\right]$ over $\overline{S}$, and we have
\begin{equation} \label{mod_heightI}
h(\mathcal{X}/\overline{S}) = \int_{\overline{S}} \,\, \left[c_1(\ll_S)\right] \, . 
\end{equation}
For each $s \in \overline{S}$ we denote by $\overline{G}_s$ the polarized weighted graph associated to the fiber of~$\pi$ above~$s$.
Following the proof of \cite[Lemma~XI.9.17]{acg} or \cite[Lemma~2.11]{bghdj} we may deduce from the estimate (\ref{lambda_estimate}) in Theorem~\ref{lambda_asympt_intro} that $ \frac{\partial \overline{\partial}}{\pi i} \lambda$ determines a current over $\overline{S}$, and moreover for $p \in \overline{S}$ a point and for $\dd \isom U$ with $0 \mapsto p$  a small coordinate neighborhood of $p$  we have that the residue
\begin{equation}
 \lim_{\epsilon \to 0} \int_{\partial \dd_\epsilon} \frac{1}{\pi i} \overline{\partial} \lambda = \lambda(\overline{G}_s) \, . 
 \end{equation}
 Write
\begin{equation}
 \lambda(\mathcal{X}/\overline{S}) =  \sum_{s \in |\overline{S}|} \lambda(\overline{G}_s) \, . 
\end{equation}
An application of Stokes' theorem then yields
\begin{equation} \label{stokes}
-\int_{\overline{S}} \, \, \left[ \frac{\partial \overline{\partial}}{\pi i} \lambda \right] =  \lambda(\mathcal{X}/\overline{S}) \, .
\end{equation}
We recall from (\ref{mod_height:intro}) that
\begin{equation} \label{mod_height}
h(\mathcal{X}/\overline{S})  = \frac{3g-3}{2g+1} c(\mathcal{X}/\overline{S}) +   \lambda(\mathcal{X}/\overline{S}) \, . 
\end{equation}
From (\ref{c1_overS}) we find
\begin{equation} \label{c1_overSbar}
 (8g+4) \left[c_1(\ll_S) \right] = \left[ c_1(\bb_S) \right] - (8g+4) \left[ \frac{\partial \overline{\partial}}{\pi i} \lambda  \right] \, . 
\end{equation}
We find the required equality
\begin{equation}
12(g-1) \, c(\mathcal{X}/\overline{S}) = \int_{\overline{S}} \,\, [c_1(\bb_S)]
\end{equation}
upon combining (\ref{mod_heightI}), (\ref{stokes}), (\ref{mod_height})  and (\ref{c1_overSbar}).
\end{proof}

\section{Jumps in the archimedean height} \label{ht_jump}

Before entering into the proof of Theorem~\ref{main:intro} we review here briefly the notion of height jumps as introduced by Hain in \cite{hain_normal} in the setting of the Ceresa cycle on the moduli space of curves, and as further analyzed by Brosnan and Pearlstein in a general context in \cite{bp}. 

\subsection{Biextension metric and archimedean height}
Let $\overline{S}$ be a smooth complex algebraic variety. Let $D$ be a normal crossings divisor on $\overline{S}$, and write $S = \overline{S} \setminus D$. Let $\UU$ be a weight~$-1$ polarized variation of Hodge structure over $S$, and let $\nu$ be a normal function section of the Griffiths intermediate Jacobian fibration $\jj(\UU)$ of $\UU$ over $S$. We recall that to give such a normal function is to give an element of the group of Yoneda extensions $\Ext^1(\zz,\UU)$ in the category of variations of mixed Hodge structures over $S$. 

As is explained in \cite[\S\S 6--7]{hain_biext}, there is a holomorphic line bundle $\widehat{\bb} \to \jj(\UU)$ whose underlying $\gg_\mathrm{m}$-torsor classifies symmetric biextensions over $\UU$. The line bundle $\widehat{\bb}$ is equipped with a canonical smooth hermitian metric, called the \emph{biextension metric}. By pulling back along the normal function $\nu \colon S \to \mathcal{J}(\UU)$ we obtain a smooth hermitian line bundle $\bb = \nu^*\widehat{\bb}$ on $S$. We denote by $\|\cdot\|_\bb$ the induced metric on~$\bb$. By \cite[Theorem~13.1]{hain_normal} or \cite[Theorem~8.2]{pp}, the $(1,1)$-form $c_1(\bb,\|\cdot\|_\bb)$ is semi-positive. For $\mathcal{V}$ a nowhere vanishing holomorphic local  section of $\bb$, the smooth plurisubharmonic (psh) function $h(\mathcal{V}) = -\log \|\mathcal{V}\|_\bb$ is called the \emph{archimedean height} of $\mathcal{V}$.

The height jump is a device to quantify the singularities of the metric $\|\cdot\|_\bb$ near the points of the boundary divisor $D$ of $S$ in $\overline{S}$. We suppose throughout that the monodromy operators of $\UU$ around the branches of $D$ are unipotent, and that the normal function $\nu$ is \emph{admissible} in the sense of M.\ Saito. Normal functions that arise from families of algebraic cycles over $S$ by the Griffiths Abel-Jacobi construction -- e.g., the normal function on $\mm_g$ associated to the Ceresa cycle -- are admissible. 

\subsection{General definition of the height jump}
Let $p \in \overline{S}$ be a point and let $n = \dim S$. Let  $U \isom \dd^n$ with $p \mapsto 0$ be a small coordinate neighborhood of $p$ in $\overline{S}$. Suppose that the normal crossings divisor $D$ is given by the equation $t_1 \cdots t_r = 0$ in $U$. As is shown in \cite[Theorem~81]{bp}, we may assume that the set $\bb^\times(U \setminus D)$ of nowhere vanishing holomorphic sections of $\bb$ over $U \setminus D$ is non-empty. We choose an element $\mathcal{V} \in \bb^\times(U \setminus D)$. 

As is shown in \cite[Theorem~5.37]{pe}, there exists a unique homogeneous weight one element $\mu \in \qq(x_1,\ldots,x_r)$ such that for all $m = (m_1,\ldots,m_r) \in \zz_{>0}^r$ and for all holomorphic arcs $f \colon \dd \to U$ with $f(0)=p$ and $f(\dd^*) \subset U \cap S$ and $\ord_0 f^* t_i = m_i$ for $i=1,\ldots,r$, the asymptotics
\[ f^* h(\mathcal{V}) \sim -\mu(m_1,\ldots,m_r) \log |t| \]
holds as $t \to 0$. 
For $i=1,\ldots,r$ we let $D_i$ denote the local branch of $D$ at $p$ given by the equation $t_i=0$, and let $f_i \colon \dd \to U$ be a holomorphic arc that intersects the branch $D_i$ transversally and has empty intersection with the other branches. Then similarly for all $i=1,\ldots,r$ there exists $\mu_i \in \qq$ such that the asymptotics 
\[ f_i^* h(\mathcal{V}) \sim - \mu_i \log |t| \]
holds as $t \to 0$. 

\begin{definition} \label{def:ht_jump} 
The \emph{height jump} at $p$ is the element $j(p) \in \qq(x_1,\ldots,x_r)$ uniquely determined by the equality
\begin{equation} \label{def:ht_jump_eqn}
 j(p;m_1,\ldots,m_r) = - \mu(m_1,\ldots,m_r) + \sum_{i=1}^r m_i \mu_i  \, , \quad (m_1,\ldots,m_r) \in \zz_{>0}^r \, . 
\end{equation}
As is straightforward to check, the element $j(p) \in \qq(x_1,\ldots,x_r)$ is independent of the choice of $\mathcal{V} \in \bb^\times(U \setminus D)$. Moreover when viewing $j(p)$ as an element of $\qq(x_i \, | \, i \in \mathcal{I} )$ with $\mathcal{I}$ the set of local branches of $D$ at $p$, the height jump is independent of the choice of coordinate neighborhood $U$ of $p$ as well. When viewed as an invariant of the points of $\overline{S}$, the height jump is locally constant on the natural combinatorial strata of the normal crossings divisor~$D$. Let $D^\sing$ denote the singular locus of $D$. By construction, the height jump vanishes on $\overline{S} \setminus D^\sing$.
\end{definition}

Following \cite[equation (23)]{bp} we normalize the function $h(\mathcal{V}) \colon U \setminus D \to \rr$ by putting
\[ \bar{h}(\mathcal{V}) = h(\mathcal{V}) + \sum_{i=1}^r \mu_i \log |t_i|  \, . \]
We have the following properties of the function $ \bar{h}(\mathcal{V}) $.
\begin{itemize}
\item[(i)] The function $\bar{h}(\mathcal{V}) $ extends to a continuous function $h^*(\mathcal{V})$ over $U \setminus D^\sing$ (cf.\ \cite[Theorem~24]{bp});
\item[(ii)] The function $h^*(\mathcal{V})$ extends to a psh function over $U$ (cf.\ \cite[Theorem~27]{bp}).
\end{itemize}

\subsection{The asymptotic height pairing}
In the following we keep the point $p \in \overline{S}$ fixed.
Let $I\!H^1(\UU) \subseteq H^1(U\setminus D,\UU)$ be the local intersection cohomology group of the variation $\UU$ at $p$. One of the main purposes of \cite{bp}  is to show that the height jump at $p$ can be explained in terms of a natural pairing, called the \emph{asymptotic height pairing}, on $I\!H^1(\UU)$. More precisely, for $m=(m_1,\ldots,m_r) \in \qq_{\geq 0}^r$ let 
\[ h(m_1,\ldots,m_r) \colon I\!H^1(\UU) \times I\!H^1(\UU) \longrightarrow \qq \]
be the asymptotic height pairing associated to $m$ as defined in \cite[\S6]{bp}. Let $\sing(\nu) \in H^1(U\setminus D,\UU)$ be the \emph{singularity} of the normal function $\nu$, i.e.\ the image of the normal function section $\nu \in H^0(U\setminus D,\jj(\UU))$ under the homomorphism  $H^0(U \setminus D,\jj(\UU)) \to H^1(U \setminus D,\UU)$ coming from the natural short exact sequence
\[ 0 \to \UU \to \UU \otimes_\zz \oo_S \to \mathcal{J}(\UU) \to 0   \]
of sheaves on $S$.
 Then $\sing(\nu) \in I\!H^1(\UU)$ and, for all $m=(m_1,\ldots,m_r) \in \zz_{> 0}^r$, the equality
\[ j(p;m_1,\ldots,m_r) = h(m_1,\ldots,m_r)(\sing(\nu),\sing(\nu)) \]
holds. See \cite[Theorem~22]{bp}. 

The connection with the asymptotic height pairing  leads to the following properties of the height jump (cf.\ \cite[Proposition~140, Corollary~13 and the remarks at the end of \S1.2]{bp}):
\begin{itemize}
\item[(iii)] The function $j(p)$ extends to a continuous function over $\rr_{\geq 0}^r$;
\item[(iv)] For all $m=(m_1,\ldots,m_r) \in \rr_{\geq 0}^r$ one has $j(p;m_1,\ldots,m_r) \geq 0$;
\item[(v)] The following assertions are equivalent:
\begin{itemize}
\item the height jump $j(p)$ vanishes identically;
\item the singularity $\sing(\nu)$ of the normal function $\nu$ vanishes in $I\!H^1(\UU)$;
\item the plurisubharmonic function from (ii) is locally bounded at $p$.
\end{itemize}
\end{itemize}
Item (iv) is proved independently in \cite[Theorem~1.4]{bghdj} using different techniques. 

\subsection{The Hain-Reed line bundle} The set-up as discussed above generalizes in a straightforward manner to the setting of orbifolds. The case that we are interested in in this paper -- and which is discussed at length in \cite{hain_normal}  -- arises as follows. The starting point is the tautological variation of Hodge structures $\mathcal{H}$ of weight $-1$ on the moduli orbifold $\mm_g$ of curves of genus~$g$, where $g \geq 2$. The fiber of $\mathcal{H}$ at the moduli point given by a smooth projective connected curve $C$ of genus~$g$ is given by the first homology group $H_1(C,\zz)$, and the polarization is determined by the intersection pairing. 

Starting from $\mathcal{H}$ we form the weight $-1$ variation of Hodge structure $\bigwedge^3 \mathcal{H}(-1)$. The intersection pairing gives rise to a morphism of variations $\mathcal{H} \to  \bigwedge^3 \mathcal{H}(-1)$, and we let $\UU$ denote the cokernel of this morphism. When applied to the Ceresa cycle $C-C^-$ in the Jacobian of $C$, Griffiths' generalization of the Abel-Jacobi map gives rise to a normal function section $\nu \colon \mm_g \to \mathcal{J}(\UU)$ of the Griffiths intermediate Jacobian fibration $\mathcal{J}(\UU)$ of $\UU$ over $\mm_g$. 

The Hain-Reed line bundle $\bb$ is the smooth hermitian line bundle $\nu^*\widehat{\bb}$ on $\mm_g$ that arises from pulling back the symmetric biextension line bundle $\widehat{\bb}$ over $\mathcal{J}(\UU)$ along the normal function $\nu$. The compactification with normal crossings boundary divisor that we consider is the Deligne-Mumford compactification $\overline{\mm}_g$ that arises by adding in the stable curves of genus $g$.

 Items (i), (ii) and (v) above, when specialized to the setting of the normal function determined by the Ceresa cycle, give rise to items (i)--(iii) mentioned in \S\ref{sec:BP_work}. 
 
 When $p$ is a point of $\overline{\mm}_g$, represented by a stable curve $C$, in order to study the height jump at $p$ we will work with a versal analytic deformation space of $C$. We think of such a versal analytic deformation space as a coordinate neighborhood of $p$ on the orbifold~$\overline{\mm}_g$.

\section{Equality of  height jump and  slope}  \label{lambda_asympt_gives_jump}

In this final section we derive Theorem~\ref{main:intro} from Theorem~\ref{lambda_asympt_intro}. 
\begin{proof}[Proof of Theorem~\ref{main:intro}] 
Let $p \in \overline{\mm}_g$ be a point. Let $C$ be the stable curve corresponding to $p$, and let $(G, \mathbf{q})$ be the polarized (unweighted) dual graph of $C$. We repeat that in reality, instead of working on the orbifold $\overline{\mm}_g$ we are working on a versal analytic deformation space $U$ of the stable curve $C$, which we view as an open coordinate neighborhood of $p$ on $\overline{\mm}_g$. 

Let $\mathcal{I}$ denote the set of local branches of $\varDelta$ at $p$ and let $E$ be the set of edges of $G$. We fix an isomorphism $U \isom \dd^{3g-3}$ with $p \mapsto 0$, where we suppose that $\varDelta$ is given by the equation $t_1 \cdots t_r = 0$. The isomorphism $U \isom \dd^{3g-3}$ determines a bijection $\{1,\ldots,r \} \isom \mathcal{I}$ and hence a bijection $\{1,\ldots,r\} \isom E$ by composing with the canonical bijection $\mathcal{I} \isom E$. We write $D_i$ for the $i$-th local branch in $\mathcal{I}$, and $e_i$ for the $i$-th edge in $E$. 

For $i=1,\ldots,r$ we  let $f_i \colon \dd \to U$ be a holomorphic arc that intersects the branch $D_i$ transversally and has empty intersection with the other branches. We let $(G_i, \mathbf{q}_i)$ denote the polarized dual graph associated to the generic point of $D_i$. For each $i=1,\ldots, r$ we have that $(G_i, \mathbf{q}_i)$ is the polarized graph obtained by contracting all edges in $G$ except $e_i$, and endowing the resulting graph with the pushforward polarization.

Next we fix a tuple  $m=(m_1,\ldots,m_r) \in \zz_{>0}^r$. We let $f \colon \dd \to \overline{\mm}_g$ be a holomorphic arc with $f(0) = p$ and $f(\dd^*) \subset \mm_g$ such that for $i=1,\ldots,r$ we have $\ord_0 f^*(D_i) = m_i$.  Then endowing $(G, \mathbf{q})$ with edge lengths determined by $m$ gives the polarized weighted dual graph naturally associated to the stable curve over $\dd$ obtained by pulling back the universal family over $U$ along the map $f$. 

As was noted in \S\ref{ht_jump}, by \cite[Theorem~81]{bp} the set $\bb^\times(U \setminus \varDelta)$ of nowhere vanishing holomorphic sections of $\bb$ over $U \setminus \varDelta$ is non-empty. We choose an element $\mathcal{V} \in \bb^\times(U \setminus \varDelta)$. 
Unwinding Definition~\ref{def:ht_jump} of the height jump $j(p)$ at $p$ we obtain the asymptotic
\begin{equation} \label{jp_asympt} j(p)(m_1,\ldots,m_r)  \log|t| \sim - f^*  \left( \log \| \mathcal{V} \|_\bb \right) + \sum_{i=1}^r m_i f_i^* \left( \log \| \mathcal{V} \|_\bb \right) 
\end{equation}
as $t \to 0$ over $\dd$.

Let $\ll$ be the determinant of the Hodge bundle on $\mm_g$ equipped with its Hodge metric $\|\cdot\|_{\mathrm{Hdg}}$. 
Let $\pi \colon \overline{\mathcal{C}}_g \to  \overline{\mm}_g$ be the universal stable curve and let $\overline{\ll} = \det \pi_* \omega_{\overline{\mathcal{C}}_g /  \overline{\mm}_g}$.
Let $\phi \colon \bb \isom \ll^{\otimes 8g+4}$ be the isomorphism of holomorphic line bundles fixed in \S\ref{sec:connection_lambda}.

We consider the nowhere vanishing holomorphic section  $\phi_*(\mathcal{V})$ of $\ll^{\otimes 8g+4}$ over $U \setminus \varDelta$. For each $i=1,\ldots,r$ we let $a_i \in \zz$  denote the vanishing multiplicity of $ \phi_*(\mathcal{V})$   along the branch $D_i$ when viewed as a meromorphic section of $\overline{\ll}^{\otimes 8g+4}$ over $U$.  By Proposition~\ref{mumford}(iii) we have the asymptotics
\begin{equation} \label{pull_Hodge_i}
 f_i^* \left( \log \|\phi_*(\mathcal{V})\|_{\Hdg} \right) \sim a_i \log |t| + (4g+2) \log \det \Im \varOmega(t) 
\end{equation}
for $i=1,\ldots, r$, and  
\begin{equation} \label{pull_Hodge}
 f^* \left( \log \|\phi_*(\mathcal{V})\|_{\Hdg} \right) \sim a \log|t| + (4g+2) \log \det \Im \varOmega(t) \, ,
\end{equation}
where $a = \sum_{i=1}^r a_i m_i$.
Combining (\ref{pull_Hodge_i}) and (\ref{pull_Hodge}) leads to the asymptotic
\begin{equation} \begin{split} \label{good_metric}
 -f^* \left( \log \|\phi_*(\mathcal{V})\|_{\Hdg} \right) & + \sum_{i=1}^r m_i f_i^* \left( \log \|\phi_*(\mathcal{V})\|_{\Hdg} \right)\\ 
&  \sim  (4g+2) (-1 + \sum_{i=1}^r m_i) \log \det \Im \varOmega(t) \, . \end{split} 
\end{equation}
For $u \in U \setminus \varDelta$ we denote by $X_u$ the fiber above $u$ in the  universal family over $U$. Combining (\ref{lambda_new_def}), (\ref{jp_asympt})  and (\ref{good_metric})  we find 
\begin{equation}
 \begin{split} j(p) & (m_1, \ldots,m_r)  \log|t| \sim  \\
 & \sim (8g+4) \left(-  \lambda(X_{f(t)}) + \sum_{i=1}^r m_i \lambda(X_{f_i(t)}) \right) \\
 & \hspace{2cm}  + (4g+2) (-1 + \sum_{i=1}^r m_i) \log \det \Im \varOmega(t) \, . \end{split}
\end{equation}
Invoking Theorem \ref{lambda_asympt_intro} this leads to
\begin{equation} \begin{split} j(p) & (m_1, \ldots,m_r)  \log|t| \sim  \\
 & \sim (8g+4) \left( \lambda(G, \mathbf{q};m_1,\ldots,m_r) - \sum_{i=1}^r  \lambda(G_i, \mathbf{q}_i;m_i) \right) \log|t| 
\, . \end{split}
\end{equation}
We conclude that 
\begin{equation} \label{almost_slope}
 j(p)(m_1,\ldots,m_r) = (8g+4) \left( \lambda(G, \mathbf{q};m_1,\ldots,m_r) - \sum_{i=1}^r \lambda(G_i, \mathbf{q}_i;m_i) \right) \, . 
\end{equation}
Let $i \in \{1,\ldots, r\}$. The polarized graph $(G_i, \mathbf{q}_i)$ is a loop graph of genus~$g$ based on a single vertex if $f_i(0) \in \varDelta_0$, and an edge segment with vertices of genera $h$ and $g-h$ if $f_i(0) \in \varDelta_h$ for $h \in \{1,\ldots,[g/2] \}$. 

By Examples~\ref{ex:loop} and~\ref{ex:segment} we find
\begin{equation} (8g+4) \lambda(G_i, \mathbf{q}_i;m_i) = g \, m_i \, , \quad f_i(0) \in \varDelta_0 \, , 
\end{equation}
and 
\begin{equation}
 (8g+4) \lambda(G_i, \mathbf{q}_i;m_i) = 4h(g-h) \, m_i \, , \quad f_i(0) \in \varDelta_h \, , \quad h \in \{1,\ldots,[g/2] \} \, . 
\end{equation}

Write $\overline{G}$ for the polarized weighted graph $(G,\mathbf{q})$ with edge lengths given by the vector~$m=(m_1,\ldots,m_r)$. As for each $i=1,\ldots, r$ the polarized graph $(G_i, \mathbf{q}_i)$ is the polarized graph obtained by contracting all edges in $G$ except $e_i$ and endowing the resulting graph with the polarization induced from $\mathbf{q}$, we conclude that
\begin{equation} \label{second_term} (8g+4) \sum_{i=1}^r  \lambda(G_i, \mathbf{q}_i; m_i) = g \, \delta_0(G) + \sum_{h=1}^{[g/2]} 4h(g-h) \, \delta_h(\overline{G}) \, . 
\end{equation}
Here as before $\delta_0(G)$ denotes the total length of the edges of $G$ that do not disconnect the graph $G$ upon removal, and  $\delta_h(\overline{G})$ for $h \in \{1,\ldots,[g/2] \}$ denotes the total length of the edges of $G$ whose removal from $G$ results in the disjoint union of a polarized graph of genus $h$ and a polarized graph of genus $g-h$. 

Combining (\ref{almost_slope}) and (\ref{second_term}) we conclude that
\[ j(p)(m_1,\ldots,m_r) = s(G, \mathbf{q};m_1,\ldots,m_r) \, . \]
Theorem~\ref{main:intro} follows since as homogeneous weight one elements of $\qq(x_1,\ldots,x_r)$, the functions $j(p)$ and $s(G, \mathbf{q})$ are both determined by their values on the elements of $\zz_{>0}^r$. 
\end{proof}

\vspace{0.5cm}
\end{document}